\documentclass[times,sort&compress,3p]{elsarticle}
\usepackage[labelfont=bf]{caption}
\renewcommand{\figurename}{Fig.}

\usepackage{amsmath,amsfonts,amssymb,amsthm,booktabs,color,epsfig,graphicx,hyperref,url,cancel}
\usepackage[subrefformat=parens]{subcaption}

\theoremstyle{plain}
\newtheorem{theorem}{Theorem}
\newtheorem{proposition}{Proposition}
\newtheorem{lemma}{Lemma}

\theoremstyle{definition}

\newcommand{\tr}{\mathrm{tr}}
\newcommand{\diag}{\mathrm{diag}}

\newcommand{\pd}[1]{\partial_{#1}}

\newcommand{\KL}[2]{\mathrm{KL} \left( \, #1 \,\middle|\!\middle|\, #2 \, \right)}
\newcommand{\LB}{\Delta_{\mathrm{LB}}}

\begin{document}

\begin{frontmatter}

\title{Enriched standard conjugate priors and the right invariant prior for Wishart distributions}  

\author[1]{Hidemasa Oda \corref{mycorrespondingauthor}}
\author[1]{Fumiyasu Komaki}

\address[1]{Department of Mathematical Informatics, Graduate School of Information Science and Technology, The University of Tokyo, 7-3-1 Hongo, Bunkyo-ku, Tokyo 113-8656, Japan}

\cortext[mycorrespondingauthor]{Corresponding author. Email address: \url{hidemasa_oda@mist.i.u-tokyo.ac.jp}}

\begin{abstract}
  The prediction of the variance-covariance matrix of the multivariate normal distribution is important in the multivariate analysis.
  We investigated Bayesian predictive distributions for Wishart distributions under the Kullback--Leibler divergence.
  The conditional reducibility of the family of Wishart distributions enables us to decompose the risk of a Bayesian predictive distribution.
  We considered a recently introduced class of prior distributions, which is called the family of enriched standard conjugate prior distributions, and compared the Bayesian predictive distributions based on these prior distributions.
  Furthermore, we studied the performance of the Bayesian predictive distribution based on the reference prior distribution in the family
  and showed that there exists a prior distribution in the family that dominates the reference prior distribution.
  Our study provides new insight into the multivariate analysis when there exists an ordered inferential importance for the independent variables.
\end{abstract}

\begin{keyword}
Conditional Reducibility
\sep
Objective Bayes
\sep
Principle of Equivariance
\sep
Sample Variance-Covariance Matrix
\sep
Statistical Decision Theory
\sep
Wishart Distribution
\MSC[2020] Primary 62C10 \sep
Secondary 62H12
\end{keyword}

\end{frontmatter}

\section{Introduction}

The estimation of the variance-covariance matrix of the multivariate normal distribution has a long history.
Estimates of variance-covariance matrices are required at regression analysis, discriminant analysis, principal component analysis, and factor analysis (see \cite{press2005applied}).
The sample variance-covariance matrix $\hat{\Sigma}$ is the most traditional estimator of the variance-covariance matrix $\Sigma$ of the multivariate normal distribution.
The Stein-type shrinkage estimator is another approach for estimating the variance-covariance matrix (see \cite{ledoit2004well}).
Bayesian approaches for estimating the variance-covariance matrix under quadratic loss or entropy loss have also been investigated (see \cite{leonard1992bayesian, yang1994estimation, sun2007objective, hsu2012bayesian, sun2007objective} as examples).

This study is based on the framework of the Bayesian statistical decision theory.
We predict the distribution of the sample variance-covariance matrix $\hat{\Sigma}'$ of a future sample of size $N$ by observing the sample variance-covariance matrix $\hat{\Sigma}$ of the current sample of size $M$.
We use the Kullback--Leibler divergence as a loss function (see \cite{aitchison1975goodness}).
This loss function is non-negative, convex, and lower semi-continuous, which meets the basic needs for developing the statistical decision theory (see \cite{wald1949statistical, lecam1955extension}).
We develop prior distributions that lead to better prediction under the Kullback--Leibler loss.

The estimation of the principal sub-matrix of the variance-covariance matrix or its inverse matrix is important (see Appendix 1 in \cite{dawid1973marginalization} for example).
Several studies have shown that shrinkage prior distributions are helpful for prediction problems (see \cite{komaki2001shrinkage, george2006improved} for normal distributions, \cite{komaki2004simultaneous} for Poisson distributions, \cite{komaki2009bayesian} for two-dimensional Wishart distributions, and \cite{9429230} for complex-valued auto-regressive processes).
An attractive family of prior distributions, which is called the family of enriched standard conjugate prior distributions, for a specific hierarchical grouping of the parameters of the Wishart distribution was introduced in \cite{consonni2003enriched}.
These prior distributions are useful when there exists an ordered inferential importance for the independent variables in the multivariate analysis.
In this study, we show that the square root of the ratio of a prior distribution in this family to the Jeffreys prior distribution is an eigenfunction of the Laplace--Beltrami operator (see Proposition \ref{proposition:main_asymptotic}).
The family of prior distributions has a geometric interpretation in the statistical model because the Laplace--Beltrami operator is independent of the choice of parameterization.
We use the conditional reducibility (see \cite{consonni2001conditionally, consonni2003enriched}) of the family of Wishart distributions, allowing us to decompose the risks of Bayesian predictive distributions into several terms according to the parameter importance, and minimize each term on each parameter space.

One of the focuses in our study is the reference prior distribution, which is originally intended to describe a prior distribution that is not informative of the observations (see \cite{bernardo1979reference}).
That is, the reference prior distribution is defined as the prior distribution that maximizes the mutual information between parameters and observations (see \cite{berger2009formal}).
We note that the definition of the original reference prior distribution in terms of mutual information is invariant under reparameterization.
In this study, we investigate the reference approach, which is a procedure for constructing a non-informative prior distribution with respect to the order of the parameters, according to their inferential importance (see \cite{berger1992ordered, consonni2003enriched}).
The reference approach is advantageous because of its ability to distinguish between parameters of interest and nuisance parameters.
In this approach, we note that the definition of the reference prior distribution may depend on the specific order of the parameters.
In this study, we investigate a prior distribution in the family of enriched standard conjugate prior distributions that dominates the reference prior distribution (see Theorem \ref{theorem:main_exact}).

The remainder of the paper is organized as follows.
In Section \ref{section:bayesian_prediction_for_wishart_distribution}, a brief explanation of the decision-theoretic Bayesian prediction is provided using the Kullback--Leibler divergence.
In Section \ref{section:conditional_reducibility}, the definition of conditional reducibility for families of probability distributions are given.
The conditional reducibility of the family of Wishart distributions are discussed in Section \ref{section:conditional_reducibility_wishart}.
In Section \ref{section:main_asymptotic}, we investigate a prior distribution in the family of Wishart distributions that asymptotically dominates the reference prior distribution as $M \to \infty$.
The proposed prior distribution dominates the reference prior distribution for any value of $M$ and $N$, as demonstrated in Section \ref{section:main_exact}.
Moreover, we provide a geometric interpretation of the dominance of the prior distributions in the family of Wishart distributions over the Jeffreys prior distribution in Section \ref{section:geometry_of_bayesian_predictive_distribution}.
Furthermore, in Section \ref{section:relative_invariance}, we provide another interpretation using the relative invariance under the left action of the group of upper-triangular block matrices.

\section{Bayesian prediction for Wishart distributions}
\label{section:bayesian_prediction_for_wishart_distribution}

In this study, we consider the real Wishart distribution on the space $S_r^+ \left( \mathbb{R} \right)$ of positive definite symmetric $r \times r$ matrices, or the complex Wishart distribution on the space $H_r^+ \left( \mathbb{C} \right)$ of positive definite Hermitian $r \times r$ matrices.
However, the results presented in this study may be valid for other Wishart distributions associated with abstract Euclidean Jordan algebras (see \cite{faraut1994analysis} for the definition and classification of Wishart distributions).
The symmetric cone $S_r^+ \left( \mathbb{R} \right)$ or $H_r^+ \left( \mathbb{C} \right)$ is denote by $E_r^+$.
In this study, we use the notations presented in Table \ref{table:simple_euclidean_jordan_algebras},
where $r$ is the rank, $n$ is the dimension, and $d$ is the Peirce invariant of the symmetric cone $E_r^+$.
We have $n = r + r \left( r - 1 \right) d / 2$.

\begin{table}[htb]
  \caption{Classification of the symmetric cones, on which Wishart distributions are defined. $r$ is the rank, $n$ is the dimension, and $d$ is the Peirce invariant of the Symmetric cone $E_r^+$.}
  \vskip-0.3cm
  \hrule
  \smallskip
  \centering\small
  \begin{tabular}{lll}
    \label{table:simple_euclidean_jordan_algebras}
    $E_r^+$ & $n$ & $d$ \\
    $S_r (\mathbb{R})$ & $\frac{1}{2} r(r+1)$ & $1$ \\
    $H_r (\mathbb{C})$ & $r^2$ & $2$ \\
  \end{tabular}
  \hrule
\end{table}

First, we define the multivariate gamma function $\Gamma_r \left( \mu \right)$ to describe the Wishart distribution of rank $r$.
We use $\left| x \right|$ to denote the determinant of matrix $x$.
The multivariate gamma function is defined as
\begin{align*}
  \Gamma_r \left( \mu \right)
  &:=
  \int_{E_r^+} \left| x \right|^{\mu - \frac{n}{r} } \exp \left( - \tr \left( x \right) \right) dx
\end{align*}
for $\mu > \left( r - 1 \right) d / 2$,
where $dx$ is the Lebesgue measure on $E_r^{+}$.
The $(i+1)$-th derivative of the multivariate log-gamma function $\log \Gamma_r \left( \mu \right)$ for $\mu > \left( r - 1 \right) d / 2$ is called the multivariate polygamma function $\psi^{(i)}_r \left( \mu  \right)$ of order $i$.
Particularly, $\psi_r \left( \mu \right) := \psi_r^{(0)} \left( \mu \right) = \frac{d}{d\mu} \log \Gamma_r \left( \mu \right)$ is called a multivariate digamma function.
$(-1)^{i+1} \psi_r^{(i)} (\mu)$ is positive and strictly decreasing to zero if $i > 0$.
We have
\begin{align}
  \psi_r^{(i)} \left( \mu \right) = \sum_{k=1}^r \psi^{(i)} \left( \mu - \left( k - 1 \right) \frac{d}{2} \right)
  \label{formula:polygamma}
  \,,
\end{align}
where $\psi^{(i)} \left( \mu \right) := \psi_1^{(i)} \left( \mu \right)$ is the usual univariate polygamma function for $\mu > 0$.

The Wishart distribution of rank $r$ is a probability distribution on the space $E_r^{+}$.
As opposed to using the common parameterization of the Wishart distribution $W^*_r \left( M, \Sigma \right)$ with its degree of freedom $M > r - 1$ and expectation $M \Sigma$, we parameterize the family of Wishart distributions $W_r \left( \mu, \xi \right)$ using $\mu > \left( r - 1 \right) d / 2$ and $\xi \in \Xi := E_r^+$ based on the works \cite{letac1989characterization, casalis1996lukacs}
(i.e., the parameter $\mu$ is defined as $M / d$, and the parameter $\xi$ is defined as $\Sigma^{-1} / d$).
We use $\langle x \mid y \rangle := \mathrm{tr} (xy)$ to denote the usual inner product of symmetric or Hermitian matrices $x$ and $y$.
The probability distribution $p^{\mu} \left( x \mid \xi \right) \,dx$ of the Wishart distribution $W_r \left( \mu, \xi \right)$ is
\begin{align}
  p^{\mu} \left( x \mid \xi \right) \,dx
  &= \frac{ \left| x \right|^{ \mu - \frac{n}{r} } }{ \Gamma_r (\mu)} | \xi |^{\mu} \exp \left( - \langle \xi \mid x \rangle \right) \,dx
  \,,
  \label{definition:wishart_distribution}
\end{align}
where $x \in \mathcal{X} := E_r^{+}$, $\mu > \left( r - 1 \right) d / 2$ and $\xi \in \Xi = E_r^{+}$.
We write $X \sim W_r \left( \mu, \xi \right)$ if a random variable $X$ that takes its value on $E^{+}_r$ is distributed according to the Wishart distribution $W_r \left( \mu, \xi \right)$.
If $\mu$ is a fixed parameter, (\ref{definition:wishart_distribution}) shows that the family $\{ W_r \left( \mu, \xi \right) \}_{\xi \in \Xi}$ of Wishart distributions is a natural exponential family with a natural parameter $\xi$ (see \cite{letac1989characterization}).
If $X \sim W_r \left( \mu, \xi \right)$, then $E \left[ X \right] = \mu \xi^{-1}$ and $E \left[ \log \left| X \right| \right] = \psi_r \left( \mu \right) - \log \left| \xi \right|$.

Let us consider the problem of constructing a predictive distribution $\delta \left( y \mid x \right) \,dy$ of a random variable $Y$ that takes its value in the sample space $\mathcal{Y} := E_r^{+}$ by observing a random variable $X$ that takes its value on the sample space $\mathcal{X} := E_r^{+}$.
We assume that the random variables $X$ and $Y$ are distributed independently according to $W_r \left( \mu, \xi \right)$ and $W_r \left( \nu, \xi \right)$, respectively,
where the values of the parameters $\mu > \left( r - 1 \right) d / 2$ and $\nu > \left( r - 1 \right) d / 2$ are known in advance.
This situation naturally arises when we predict the sample variance-covariance matrix of the multivariate normal distribution.
In this situation, parameters $\mu$ and $\nu$ correspond to the size $M$ and $N$ of the current and future observations, respectively.

We regard each predictive distribution $\delta \left( y \mid x \right) \,dy$ as a non-randomized statistical decision function.
The decision space is defined as all probability distributions on $\mathcal{Y}$, and it is not a finite-dimensional space.
We aim to construct a predictive distribution that is not dissimilar from the true distribution.
Divergences were introduced to measure the dissimilarity between two probability distributions (see \cite{renyi1961measures, amari2009alpha, nielsen2013chi}).
For example, for probability distributions $P = p \left( y \right) dy$ and $Q = q \left( y \right) dy$,
the Kullback--Leibler divergence is defined as
\begin{align*}
  \KL{P}{Q} := \int_{\mathcal{Y}} \log \frac{ dP }{ dQ } dP
  = \int_{\mathcal{Y}} p \left( y \right) \log \frac{ p \left( y \right) }{ q \left( y \right) } dy
  = E_Y \left[ \log \frac{ p \left( Y \right) }{ q \left( Y \right) } \right]
  \,,
\end{align*}
where $E_Y$ denotes the expectation over the random variable $Y \sim P$.
In this study, we define the loss function $L$ as
\begin{align*}
  L \left( \xi, Q \right) := \KL{ p^{\nu} \left( y \mid \xi \right) dy }{ q \left( y \right) dy }
\end{align*}
for a probability distribution $Q = q \left( y \right) dy$,
where $p^{\nu} \left( y \mid \xi \right) dy$ is the true distribution for the random variable $Y \sim W_r \left( \nu, \xi \right)$.
A predictive distribution $\delta$ is a mapping that associates a probability distribution $\delta \left( y \mid x \right) dy$ on the space $\mathcal{Y}$ for each element $x$ in space $\mathcal{X}$.
Therefore, a predictive distribution $\delta$ corresponds to a non-randomized statistical decision function in statistical decision theory, where each decision corresponds to a probability distribution on the sample space $\mathcal{Y}$ (see Section 2 in \cite{robert2007bayesian}).

Thus, we propose to construct a predictive distribution $\delta$ for a random variable $Y \sim W_r \left( \nu, \xi \right)$ such that its risk $R^{\mu, \nu} \left( \xi , \delta \right)$ is small for all $\xi \in \Xi$.
The risk of a predictive distribution $\delta$ is defined as the average
\begin{align}
  R^{\mu, \nu} \left( \xi, \delta \right) := \int_{\mathcal{X}} L \left( \xi, \delta \left( y \mid x \right) dy \right) p^{\mu} \left( x \mid \xi \right) dx
  =
  E_{X, Y}
  \left[
    \log
    \frac{
      p^{\nu} \left( Y \;\middle|\; \xi \right)
    }{
      \delta \left( Y \;\middle|\; X \right)
    }
  \right]
  \label{definition:R}
\end{align}
of the loss $L \left( \xi, \delta \left( y \mid x \right) dy \right)$ over the distribution $X \sim W_r \left( \mu, \xi \right)$.
The risk $R^{\mu, \nu} \left( \xi, \delta \right)$ is non-negative, and it is zero if and only if $\delta \left( y \mid x \right) dy$ = $p^{\nu} \left( y \mid \xi \right) dy$ almost everywhere.

A predictive distribution $\delta$ is minimax if
$\sup_{\xi \in \Xi} R^{\mu, \nu} \left( \xi, \delta \right) \leq \sup_{\xi \in \Xi} R^{\mu, \nu} \left( \xi, \delta' \right)$
for any predictive distribution $\delta'$.
A predictive distribution $\delta$ dominates a predictive distribution $\delta'$
if $R^{\mu, \nu} \left( \xi, \delta \right) \leq  R^{\mu, \nu} \left( \xi ,\delta' \right)$ for any $\xi$ and $R^{\mu, \nu} \left( \xi, \delta \right) <  R^{\mu, \nu} \left( \xi ,\delta' \right)$ for some $\xi$.
Moreover, a predictive distribution $\delta$ of $Y$ is admissible if no predictive distribution $\delta'$ of $Y$ dominates the predictive distribution $\delta$.

The integrated risk (average risk) of a predictive distribution $\delta$ with respect to a possibly improper prior distribution $\pi \left( \xi \right) d \xi$ on the parameter space $\Xi$ is defined as the average
\begin{align*}
  R^{*\mu, \nu} \left( \pi, \delta \right)
  :=
  \int_{\Xi} R^{\mu, \nu} \left( \xi, \delta \right) \pi \left( \xi \right) d\xi
\end{align*}
of the risk $R^{\mu, \nu} \left( \xi, \delta \right)$ over the distribution $\pi \left( \xi \right) d \xi$.

The Bayesian predictive distribution $\delta^{\mu, \nu}_{\pi}$ based on a possibly improper prior distribution $\pi \left( \xi \right) d \xi$ on the parameter space $\Xi$ is defined as the weighted average
\begin{align*}
  \delta^{\mu, \nu}_{\pi} \left( y \mid x \right) dy
  :=
  \left(
    \int_{\Xi} p^{\nu} \left( y \mid \xi \right) \pi^{\mu} \left( \xi \mid x \right) d\xi
  \right)
  dy
\end{align*}
of the Wishart distributions $W \left( \nu, \xi \right)$ over the posterior distribution $\pi^{\mu} \left( \xi \mid x \right) d\xi$ of the prior distribution $\pi \left( \xi \right) d \xi$, given a sample $x$ of the random variable $X \sim W \left( \mu, \xi \right)$.
A prior distribution $\pi \left( \xi \right) \,d\xi$ is minimax if the Bayesian predictive distribution based on this prior distribution is minimax.
A prior distribution $\pi \left( \xi \right) \,d\xi$ dominates another prior distribution $\pi' \left( \xi \right) \,d\xi$ if the Bayesian predictive distribution based on the prior distribution $\pi \left( \xi \right) \,d\xi$ dominates the Bayesian predictive distribution based on the prior distribution $\pi' \left( \xi \right) \,d\xi$.
A prior distribution $\pi \left( \xi \right) \,d\xi$ is admissible if the Bayesian predictive distribution based on this prior distribution is admissible.

For any predictive distribution $\delta$, we have
\begin{align*}
  R^{*\mu, \nu} \left( \pi, \delta \right) - R^{*\mu, \nu} \left( \pi, \delta^{\mu, \nu}_{\pi} \right)
  =
  \int_{\mathcal{X}} \KL{ \delta^{\mu, \nu}_{\pi} \left( y \mid x \right) dy }{ \delta \left( y \mid x \right) dy } m^{\mu}_{\pi} \left( x \right) dx
  \geq
  0
  \,,
\end{align*}
where $m_{\pi}^{\mu} \left( x \right) := \int_{\Xi} p^{\mu} \left( x \mid \xi \right) \pi \left( \xi \right) d\xi$ is the marginal distribution of the random variable $X \sim W_r \left( \mu, \xi \right)$ with respect to the prior distribution $\pi \left( \xi \right) d\xi$.
Therefore, if $R^{*\mu, \nu} \left( \pi, \delta^{\mu, \nu}_{\pi} \right) < +\infty$, the Bayesian predictive distribution $\delta^{\mu, \nu}_{\pi}$ is the Bayesian procedure with respect to the prior distribution $\pi \left( \xi \right) d\xi$, and the integrated risk $R^{*\mu, \nu} \left( \pi, \delta^{\mu, \nu}_{\pi} \right)$ corresponds to the Bayes risk $R^{*\mu, \nu} \left( \pi \right) := \inf_{\delta} R^{*\mu, \nu} \left( \pi, \delta \right)$ (see \cite{aitchison1975goodness}).

\section{Conditional reducibility of probability distributions}
\label{section:conditional_reducibility}

In this section, we state the definition of conditional reducibility for families of probability distributions.
Natural exponential families with quadratic variance functions were investigated in \cite{morris1982natural, morris1983natural}.
Among them, (i) natural exponential families with homogeneous variance functions (see \cite{casalis1991familles}), such as the Wishart distribution, and (ii) natural exponential families with simple quadratic variance functions (see \cite{casalis19962}), such as the multinomial distribution and the negative multinomial distribution, are of practical importance.
These families of probability distributions exhibit a property called conditional reducibility, which was discovered in \cite{consonni2001conditionally, consonni2003enriched}.
In this section, we explain that the conditional reducibility enables us to decompose the risk of a Bayesian predictive distribution into the parts that are related to the corresponding groups of parameters.

A family $\{ p \left( x \mid \xi \right) dx \}_{\xi \in \Xi}$ of probability distributions on a sample space $\mathcal{X}$ is called conditionally $h$-reducible
if there exists an isomorphism
\begin{align*}
  \mathcal{X} \xrightarrow{\cong} \mathcal{X}^{(1)} \times \cdots \times \mathcal{X}^{(h)}
  \;,\;
  x \mapsto \left( x^{(1)}, \ldots, x^{(h)} \right)
\end{align*}
of sample spaces and an isomorphism
\begin{align*}
  \Xi \xrightarrow{\cong} \Phi = \Phi^{(1)} \times \cdots \times \Phi^{(h)}
  \;,\;
  \xi \mapsto \phi = \left( \phi^{(1)}, \ldots, \phi^{(h)} \right)
\end{align*}
of parameter spaces such that
\begin{align*}
  p \left( x \;\middle|\; \xi \right) dx
  =
  \prod_{i=1}^h
  p^{(i)} \left( x^{(i)} \;\middle|\; x_{(i-1)}, \phi^{(i)} \right) dx^{(i)}
  \,,\;
  d\xi = \prod_{i=1}^h d\phi^{(i)}
  \,,
\end{align*}
for $i \in \{ 1, \ldots, h \}$,
where the superscript $(i)$ refers to the $i$-th part of the partition and the subscript $(i)$ refers to the parts up to the $i$-th part of the partition (i.e., $x_{(i)} := \left( x^{(1)}, \ldots, x^{(i)} \right)$ and $\phi_{(i)} := \left( \phi^{(1)}, \ldots, \phi^{(i)} \right)$).

We consider the problem of constructing a predictive distribution of a random variable $Y$ by observing a random variable $X$.
Suppose that the random variable $X$ is distributed according to the distribution $p^{\mu} \left( x \;\middle|\; \xi \right) dx$, which belongs to a family $\{ p^{\mu} \left( x \;\middle|\; \xi \right) dx \}_{\xi \in \Xi}$ of probability distributions, and that the random variable $Y$ is distributed according to the distribution $p^{\nu} \left( y \;\middle|\; \xi \right) dy$, which belongs to a family $\{ p^{\nu} \left( y \;\middle|\; \xi \right) dy \}_{\xi \in \Xi}$ of probability distributions.
$\mu$ and $\nu$ are used only for distinguishing the two families.
We use the risk $R^{\mu, \nu} \left( \xi, \delta \right) := E_{X, Y} \left[ \log \frac{ p^{\nu} \left( Y \;\middle|\; \xi \right) }{ \delta \left( Y \;\middle|\; X \right) } \right]$ for evaluating a predictive distribution $\delta \left( y \;\middle|\; x \right) dy$; see (\ref{definition:R}) in Section \ref{section:bayesian_prediction_for_wishart_distribution}.
Suppose that two families are conditionally $h$-reducible with respect to the isomorphism $\Xi \cong \Phi = \Phi^{(1)} \times \cdots \times \Phi^{(h)}$ of parameter spaces.

First, we consider the estimative distribution (i.e., plugin distribution) for the prediction of the random variable $Y$.
Let $\bar{\phi} = \bar{\phi}^{\mu} \left( X \right)$ be an estimator of the true parameter $\phi$.
We use the estimative distribution, which is expressed as
\begin{align*}
  \delta^{\mu, \nu}_{\bar{\phi}} \left( y \;\middle|\; x \right) dy
  :=
  p^{\nu} \left( y \;\middle|\; \bar{\phi}^{\mu} \left( x \right) \right) dy
  \,,
\end{align*}
as a predictive distribution.
Suppose that the parameter $\phi^{(i)}$ is estimated only from the observation $x_{(i)}$ (i.e., $\bar{\phi}^{(i)} := \bar{\phi}^{(i) \mu} \left( x_{(i)} \right)$).
The family $\{ p^{\nu} \left( y \;\middle|\; \xi \right) dy \}_{\xi \in \Xi}$ is conditionally $h$-reducible.
Therefore, the predictive distribution $\delta^{\mu, \nu}_{\bar{\phi}}$ decomposes into the product
\begin{align*}
  \delta^{\mu, \nu}_{\bar{\phi}} \left( y \;\middle|\; x \right) dy
  =
  \prod_{i=1}^h
  \delta^{(i) \mu, \nu}_{\bar{\phi}^{(i)}} \left( y^{(i)} \;\middle|\; y_{(i-1)},  x_{(i)} \right) dy^{(i)}
\end{align*}
of the conditional predictive distributions
\begin{align*}
  \delta^{(i) \mu, \nu}_{\bar{\phi}^{(i)}} \left( y^{(i)} \;\middle|\; y_{(i-1)}, x_{(i)} \right) dy^{(i)}
  :=
  p^{(i) \nu} \left( y^{(i)} \;\middle|\; y_{(i-1)}, \bar{\phi}^{(i) \mu} \left( x_{(i)} \right) \right) dy^{(i)}
  \,.
\end{align*}
The risk $R^{\mu, \nu} \left( \xi, \delta^{\mu, \nu}_{\bar{\phi}} \right)$ of the predictive distribution $\delta^{\mu, \nu}_{\bar{\phi}}$ is decomposed into the sum
\begin{align*}
  R^{\mu, \nu} \left( \xi,  \delta^{\mu, \nu}_{\bar{\phi}} \right)
  =
  \sum_{i=1}^h
  R^{(i) \mu, \nu} \left( \phi_{(i)}, \delta^{(i) \mu, \nu}_{\bar{\phi}^{(i)}} \right)
\end{align*}
of the risks
\begin{align*}
  R^{(i) \mu, \nu} \left( \phi_{(i)}, \delta^{(i) \mu, \nu}_{\bar{\phi}^{(i)}} \right)
  :=
  E_{X_{(i)}, Y_{(i)}}
  \left[
    \log
    \frac{
      p^{(i) \nu} \left( Y^{(i)} \;\middle|\; Y_{(i-1)}, \phi^{(i)} \right)
    }{
      \delta^{(i) \mu, \nu}_{\bar{\phi}^{(i)}} \left( Y^{(i)} \;\middle|\; Y_{(i-1)}, X_{(i)} \right)
    }
  \right]
\end{align*}
of the conditional predictive distributions $\delta^{(i) \mu, \nu}_{\bar{\phi}^{(i)}}$.

Then, we consider the Bayesian predictive distribution for the prediction of the random variable $Y$.
Let $\pi \left( \phi \right) d\phi$ be a possibly improper prior distribution on the parameter space $\Phi$.
Suppose the prior distribution $\pi \left( \phi \right) d\phi$ factorizes as the product $\pi \left( \phi \right) d\phi = \prod_{i=1}^h \pi^{(i)} \left( \phi^{(i)} \right) d\phi^{(i)}$ of prior distributions $\pi^{(i)} \left( \phi^{(i)} \right) d\phi^{(i)}$ on the parameter spaces $\Phi^{(i)}$.
The family $\{ p^{\mu} \left( x \;\middle|\; \xi \right) dx \}_{\xi \in \Xi}$ is conditionally $h$-reducible.
Therefore, the posterior distribution $\pi^{\mu} \left( \phi \;\middle|\; x \right)$ given an observation $x$ factorizes as the product $\pi^{\mu} \left( \phi \;\middle|\; x \right) d\phi = \prod_{i=1}^h \pi^{(i) \mu} \left( \phi^{(i)} \;\middle|\; x_{(i)} \right) d\phi^{(i)}$ of the posterior distributions $\pi^{(i) \mu} \left( \phi^{(i)} \;\middle|\; x_{(i)} \right) d\phi^{(i)}$ on the parameter spaces $\Phi^{(i)}$.
The Bayesian predictive distribution $\delta^{\mu, \nu}_{\pi}$ based on the prior distribution $\pi \left( \phi \right) d\phi$ is decomposed into the product
\begin{align*}
  \delta^{\mu, \nu}_{\pi} \left( y \;\middle|\; x \right) dy
  =
  \prod_{i=1}^h
  \delta^{(i) \mu, \nu}_{\pi^{(i)}} \left( y^{(i)} \;\middle|\; y_{(i-1)}, x_{(i)} \right) dy^{(i)}
\end{align*}
of the conditional Bayesian predictive distributions
\begin{align*}
  \delta^{(i) \mu, \nu}_{\pi^{(i)}} \left( y^{(i)} \;\middle|\; y_{(i-1)}, x_{(i)} \right) dy^{(i)}
  :=
  \left(
    \int_{\Phi^{(i)}}
    p^{(i) \nu} \left( y^{(i)} \;\middle|\; y_{(i-1)}, \phi^{(i)} \right)
    \pi^{(i) \mu} \left( \phi^{(i)} \;\middle|\; x_{(i)} \right)
    d\phi^{(i)}
  \right)
  dy^{(i)}
  \,.
\end{align*}
The risk $R^{\mu, \nu} \left( \xi, \delta^{\mu, \nu}_{\pi} \right)$ of the Bayesian predictive distribution $\delta^{\mu, \nu}_{\pi}$ is decomposed into the sum
\begin{align*}
  R^{\mu, \nu} \left( \xi,  \delta^{\mu, \nu}_{\pi} \right)
  =
  \sum_{i=1}^h
  R^{(i) \mu, \nu} \left( \phi_{(i)}, \delta^{(i) \mu, \nu}_{\pi^{(i)}} \right)
\end{align*}
of the risks
\begin{align}
  R^{(i) \mu, \nu} \left( \phi_{(i)}, \delta^{(i) \mu, \nu}_{\pi^{(i)}} \right)
  :=
  E_{X_{(i)}, Y_{(i)}}
  \left[
    \log
    \frac{
      p^{(i) \nu} \left( Y^{(i)} \;\middle|\; Y_{(i-1)}, \phi^{(i)} \right)
    }{
      \delta^{(i) \mu, \nu}_{\pi^{(i)}} \left( Y^{(i)} \;\middle|\; Y_{(i-1)}, X_{(i)} \right)
    }
  \right]
  \label{definition:R_i}
\end{align}
of the conditional Bayesian predictive distributions $\delta^{(i) \mu, \nu}_{\pi^{(i)}}$.

\section{Conditional reducibility of Wishart distributions}
\label{section:conditional_reducibility_wishart}

In this section, we explain the conditional reducibility of the family $\{ W_r \left( \mu, \xi \right) \}_{\xi \in \Xi}$ of Wishart distributions.
We propose the family of prior distributions and explain three important prior distributions in the family.

We consider the case $h = 2$.
We assume that a random variable $X$ is distributed according to the Wishart distribution $W_r \left( \mu, \xi \right)$ and consider a partition $\mathbf{k} := \left( k, r - k \right)$ of rank $r$.
We shall consider the parameterization $\phi = \left( \phi^{(1)}, \phi^{(2)} \right)$ such that the parameters $\phi^{(1)}$ and $\phi^{(2)}$ govern the distribution of the principal $k \times k$ part $X_1$ of the random variable $X$ and the rest of $X$ given $X_1$, respectively.
We write the parameter $\xi \in E_{r}^+$ as
$
  \begin{bmatrix}
    \xi_{1} & \xi_{1/2} \\
    { \xi_{1/2} }^* & \xi_0
  \end{bmatrix}
$
and the random variable $X$ as
$
  \begin{bmatrix}
    X_{1} & X_{1/2} \\
    { X_{1/2} }^* & X_0
  \end{bmatrix}
$
with respect to the partition $\mathbf{k}$.
The random variable $X_1$ is distributed according to the Wishart distribution $W_{k} \left( \mu, \zeta_1 \right)$, where the parameter $\zeta_1 \in E_{k}^+$ is defined as the Schur complement $\xi_1 - \xi_{1/2} \left( \xi_0 \right)^{-1} {\xi_{1/2}}^* \in E_{k}^+$ of the parameter $\xi_0 \in E_{r-k}^+$.
Note that we have $E \left[ X_1 \right] = \mu \left( \zeta_1 \right)^{-1}$, whereas $E \left[ X \right] = \mu {\xi}^{-1}$.
In the context of the reference approach, $\phi^{(1)} := \left( \zeta_1 \right)$ corresponds to the parameter of interest, and $\phi^{(2)} := \left( \xi_{1/2}, \xi_0 \right)$ corresponds to the nuisance parameter.
Note that the determinant of the Jacobian of the transformation $\xi \mapsto \phi := \left( \phi^{(1)}, \phi^{(2)} \right)$ is equal to $1$ (i.e., $d\xi = d\phi$).
In addition, we have $\left| \xi \right| = \left| \zeta_1 \right| \left| \xi_0 \right|$.
We observe that the distribution of the random variable $X_1$ only depends on the parameter $\phi^{(1)}$ and that the conditional distribution $\left( X_{1/2}, X_0 \right) \mid X_1$ of the random variable $\left( X_{1/2}, X_0 \right)$ given an observation $X_1$ only depends on the parameter $\phi^{(2)}$, exhibiting the conditional $2$-reducibility of the family of Wishart distributions with respect to the partition $\mathbf{k}$.

We consider the general case $h \geq 2$.
Suppose that a random variable $X$ is distributed according to the Wishart distribution $W_r \left( \mu, \xi \right)$ and consider a partition $\mathbf{k} = \left( k^{(1)}, \ldots, k^{(h)} \right)$ of rank $r$ into the $h$ parts $k^{(i)}$ (i.e., $r = \sum_{i=1}^h k^{(i)}$).
We set
\begin{align*}
  r_{(i)} := k^{(1)} + \cdots + k^{(i)}
\end{align*}
for $i \in \{ 1, \ldots, h \}$.
We interpret $r_{(-1)}$ as $0$ for convenience.
For each $i \in \{ 1, \ldots, h \}$, we consider the principal $r_{(i)} \times r_{(i)}$ part $X_{(i)}$ of the random variable $X$.
We shall consider the parameterization $\phi = \left( \phi^{(1)}, \ldots, \phi^{(h)} \right)$ so that for each $i \in \{ 1, \ldots, h \}$, the set of the parameters $\left( \phi^{(1)}, \ldots, \phi^{(i)} \right)$ governs the distribution of $X_{(i)}$, and the rest of the parameters $\left( \phi^{(i+1)}, \ldots, \phi^{(h)} \right)$ governs the distribution of the rest of $X$ given $X_{(i)}$.
The random variable $X_{(i)}$ is distributed according to the Wishart distribution $W_{r_{(i)}} \left( \mu, \zeta_{(i)} \right)$, where the parameter $\zeta_{(i)} \in E_{r_{(i)}}^+$ is defined as the Schur complement of the lower-right $\left( r - r_{(i)} \right) \times \left( r - r_{(i)} \right)$ part of the parameter $\xi \in E_r^+$.
Note that $X_{(h)}$ is interpreted as $X$ and that $\zeta_{(h)}$ is interpreted as $\xi$.
For $i \in \{ 2, \ldots, h \}$, 
we define $\xi_{(i-1)} \in E_{r_{(i-1)}}^+$ as the principal $r_{(i-1)} \times r_{(i-1)}$ part of the parameter $\zeta_{(i)} \in E_{r_{(i)}}^+$.
We write
\begin{align*}
  \zeta_{(i)}
  =
  \begin{bmatrix}
    \xi_{(i-1)} & \xi^{(i)}_{1/2} \\
    { \xi^{(i)}_{1/2} }^* & \xi^{(i)}_0
  \end{bmatrix}
  \in E_{r_{(i)}}^+
\end{align*}
with respect to the partition $\left( r_{(i-1)}, k^{(i)} \right)$ of rank $r_{(i)}$.
We set for $i \in \{ 2, \ldots, h \}$
\begin{align*}
  \phi_{(i)} := \left( \zeta_{(i-1)}, \xi^{(i)}_{1/2}, \xi^{(i)}_0 \right)
  \,.
\end{align*}
For convenience, we define $\phi_{(1)}$ as $\zeta_{(1)}$.
We have $\zeta_{(i)} \cong \phi_{(i)}$ because $\zeta_{(i-1)} \in E_{r_{(i-1)}}^{+}$ is the Shur complement
\begin{align*}
  \zeta_{(i-1)} = \xi_{(i-1)} - \xi^{(i)}_{1/2} \left( \xi^{(i)}_0 \right)^{-1} {\xi^{(i)}_{1/2}}^* \in E_{r_{(i-1)}}^+
\end{align*}
of $\xi^{(i)}_0 \in E_{k^{(i)}}^{+}$ in $\zeta_{(i)} \in E_{r_{(i)}}^{+}$.
We construct the parameter $\phi^{(i)}$ as
for $i \in \{ 2, \ldots, h \}$
\begin{align*}
  \phi^{(i)} := \left( \xi^{(i)}_{1/2}, \xi^{(i)}_0 \right)
  \,.
\end{align*}
For convenience, $\phi^{(1)}$ is defined as $\zeta_{(1)}$.
Note that $\phi_{(i)} \cong \left( \phi_{(i-1)}, \phi^{(i)} \right) \cong \left( \phi^{(1)}, \ldots, \phi^{(i)} \right)$ for $i \in \{ 2, \ldots, h \}$.
In particular, we have $\phi_{(h)} \cong \phi := \left( \phi^{(1)}, \ldots, \phi^{(h)} \right)$.
Moreover, note that we have $\left| \xi \right| = \prod_{i=1}^h \left| \xi^{(i)}_0 \right|$,
where $\xi^{(1)}_0$ is interpreted as $\zeta_{(1)}$.
We identify the parameters $\xi$ and $\phi$ with this one-to-one correspondence $\Xi \cong \Phi$ and denote the corresponding Wishart distribution by $W_r \left( \mu, \phi \right)$.
We write the principal $r_{(i)} \times r_{(i)}$ part $x_{(i)}$ of $x \in \mathcal{X}$ as
$
  x_{(i)}
  =
  \begin{bmatrix}
    x_{(i-1)} & x^{(i)}_{1/2} \\
    { x^{(i)}_{1/2} }^* & x^{(i)}_0
  \end{bmatrix}
  \in 
  E_{r_{(i)}}^+
$
with respect to the partition $\left( r_{(i-1)}, k^{(i)} \right)$ of rank $r_{(i)}$ and set $x^{(i)} := \left( x^{(i)}_{1/2}, x^{(i)}_{0} \right)$ for $i \in \{ 2, \ldots, h \}$.
We have an isomorphism $\mathcal{X} \xrightarrow{\cong} \mathcal{X}^{(1)} \times \cdots \times \mathcal{X}^{(h)}$, $x \mapsto \left( x^{(1)}, \ldots, x^{(h)} \right)$ of sample spaces.
The conditional distribution $\left( X^{(i)}_{1/2}, X^{(i)}_0 \right) \mid X_{(i-1)}$ of the random variable $X^{(i)} := \left( X^{(i)}_{1/2}, X^{(i)}_0 \right)$ given an observation $X_{(i-1)}$ only depends on the parameter $\phi^{(i)} \in \Phi^{(i)}$ for $i \in \{ 2, \ldots, h \}$.
The family of Wishart distributions is conditionally $h$-reducible (see Theorem 1 in \cite{consonni2003enriched}).

For $i \in \{ 1, \ldots, h \}$, we define the integer
\begin{align*}
  m^{(i)} := k^{(i)} + \frac{k^{(i)} \left( k^{(i)}  - 1 \right)}{2} d
\end{align*}
as the dimension of the symmetric cone $E^+_{k^{(i)}}$ of rank $k^{(i)}$.
The integer
\begin{align*}
  n_{(i)} := r_{(i)} + \frac{ r_{(i)} \left( r_{(i)} - 1  \right) }{ 2 } d
\end{align*}
is defined as the dimension of the symmetric cone $E^+_{r_{(i)}}$ of rank $r_{(i)}$.
Note that the integer $n_{(i)}$ corresponds to the dimension of the parameter space $\Phi_{(i)}$ for the parameter $\phi_{(i)} \cong \zeta_{(i)}$.
We interpret $m^{(0)}$ and $n_{(0)}$ as zeros for convenience.

We consider the family $\{ \pi_{\mathbf{t}} \left( \phi \right) d\phi \}_{\mathbf{t} \in \mathbb{R}^h}$ of prior distributions, defined as
\begin{align}
  \pi_{\mathbf{t}} \left( \xi \right) d\xi
  =
  \prod_{i = 1}^h \pi^{(i)}_{t^{(i)}} \left( \phi^{(i)} \right) d\phi^{(i)}
  \label{definition:relatively_invariant_priors_wishart}
\end{align}
for $\mathbf{t} = \left( t^{(1)}, \ldots, t^{(h)} \right) \in \mathbb{R}^h$,
where each distribution $\pi^{(i)}_{t^{(i)}} \left( \phi^{(i)} \right) d\phi^{(i)}$ on the parameter space $\Phi^{(i)}$ is defined as
\begin{align}
  &
  \pi^{(i)}_{t^{(i)}} \left( \phi^{(i)} \right) d\phi^{(i)}
  :=
  \left| \xi_0^{(i)} \right|^{t^{(i)}} d\xi^{(i)}_{1/2} d\xi^{(i)}_0
  \label{definition:relatively_invariant_priors_wishart_i}
  \,.
\end{align}
We interpret $d\xi^{(1)}_{1/2} d\xi^{(1)}_0$ as $d\zeta_{(1)}$.
The family $\{ \pi_{\mathbf{t}} \left( \phi \right) d\phi \}_{\mathbf{t} \in \mathbb{R}^h}$ of prior distributions (\ref{definition:relatively_invariant_priors_wishart}) is a subfamily of the family of enriched standard conjugate priors, which was introduced in \cite{consonni2001conditionally, consonni2003enriched} (see Appendix \ref{section:ESCPD}).
The posterior distribution of an enriched standard conjugate prior distribution is an enriched standard conjugate prior distribution and is calculated by the hyperparameter update of the family of enriched standard conjugate prior distributions.

There are two familiar prior distributions in the family: the Jeffreys prior distribution and the reference prior distribution.
The Fisher information matrix $H \left( \phi \right)$ of the family $\{ W_r \left( \mu, \phi \right) \}_{\phi \in \Phi}$ of Wishart distributions is of the form $\diag \left( H^{(1)} \left( \phi_{(1)} \right), \ldots, H^{(h)} \left( \phi_{(h)} \right) \right)$, and the determinant of the matrix $H^{(i)} \bigl( \phi_{(i)} \bigr)$ is as follows:
for $i \in \{ 1, \ldots, h \}$
\begin{align*}
  \left| H^{(i)} \bigl( \phi_{(i)} \bigr) \right|
  =
  a^{(i)} \bigl( \phi^{(i)} \bigr)
  \times
  b^{(i)} \bigl( \phi_{(i-1)} \bigr)
  \,,
\end{align*}
where
\begin{align*}
  a^{(i)} \bigl( \phi^{(i)} \bigr)
  :=
  \left| \xi^{(i)}_0 \right|^{ - 2 \left( \frac{r_{(i)} - 1}{2} d + 1 \right) }
  \,,\;
  b^{(i)} \bigl( \phi_{(i-1)} \bigr)
  :=
  \left| \zeta_{i-1} \right|^{ - 2 \left( \frac{k^{(i)}}{2} d  \right) }
  \,.
\end{align*}
The Jeffreys prior distribution $\pi_J \left( \phi \right) d\phi$ is calculated as
\begin{align}
  \pi_J \left( \phi \right) d\phi
  =
  \left| H \left( \phi \right) \right|^{\frac{1}{2}} d\phi
  =
  \prod_{i=1}^{h} \left| \xi^{(i)}_0 \right|^{ - \left( \frac{r-1}{2} d + 1 \right) } d\xi^{(i)}_{1/2} d\xi^{(i)}_0
  \,,
\end{align}
which belongs to the family $\{ \pi_{\mathbf{t}} \left( \phi \right) d\phi \}_{\mathbf{t} \in \mathbb{R}^h}$ with $\mathbf{t} = \mathbf{t}_J$, where
for $i \in \{ 1, \ldots, h \}$
\begin{align*}
  t^{(i)}_J :=  - \frac{n}{r} = - \left( \frac{r-1}{2} d + 1 \right)
  \in \mathbb{R}
  \,.
\end{align*}
The reference prior distribution $\pi_C \left( \phi \right) d\phi$ is calculated as
\begin{align}
  \pi_C \left( \phi \right) d\phi
  =
  \prod_{i=1}^{h} \left| a^{(i)} \bigl( \phi^{(i)} \bigr) \right|^{\frac{1}{2}} d\phi^{(i)}
  =
  \prod_{i=1}^{h} \left| \xi^{(i)}_0 \right|^{ - \left( \frac{r_{(i)}-1}{2} d + 1 \right) } d\xi^{(i)}_{1/2} d\xi^{(i)}_0
  \label{definition:reference_prior}
  \,,
\end{align}
which also belongs to the family $\{ \pi_{\mathbf{t}} \left( \phi \right) d\phi \}_{\mathbf{t} \in \mathbb{R}^h}$ with $\mathbf{t} = \mathbf{t}_C$, where
for $i \in \{ 1, \ldots, h \}$ 
\begin{align*}
  t^{(i)}_C :=  - \frac{n_{(i)}}{r_{(i)}} = - \left( \frac{r_{(i)}-1}{2} d + 1 \right)
  \in \mathbb{R}
\end{align*}
(see Section 2.3 in \cite{consonni2003enriched} for the interpretation and application of the reference prior distribution $\pi_C \left( \phi \right) d\phi$).

In this study, we consider another prior distribution
\begin{align}
  \pi_R \left( \phi \right) d\phi
  :=
  \prod_{i=1}^{h} \left| \xi^{(i)}_0 \right|^{ - \left( \frac{2 r_{(i)} - k^{(i)} - 1}{2} d + 1 \right) } d\xi^{(i)}_{1/2} d\xi^{(i)}_0
  \label{definition:right_invariant_prior}
  \,,
\end{align}
which belongs to the family $\{ \pi_{\mathbf{t}} \left( \phi \right) d\phi \}_{\mathbf{t} \in \mathbb{R}^h}$ with $\mathbf{t} = \mathbf{t}_R$, where
for $i \in \{ 1, \ldots, h \}$
\begin{align*}
  t^{(i)}_R
  :=
  - \left( \frac{n_{(i)}}{r_{(i)}} + \left( \frac{n_{(i)}}{r_{(i)}} - \frac{m^{(i)}}{k^{(i)}} \right) \right)
  =
  - \left( \frac{2 r_{(i)} - k^{(i)} - 1}{2} d + 1 \right)
  \in \mathbb{R}
  \,.
\end{align*}
Note that $t^{(1)}_R = t^{(1)}_C = - \left( \frac{k^{(1)} - 1}{2}d + 1\right) \in \mathbb{R}$.
This study aims to compare the reference prior distribution (\ref{definition:reference_prior})
and our prior distribution (\ref{definition:right_invariant_prior}).
The importance of our prior distribution $\pi_R \left( \phi \right) d\phi$ is discussed in later sections.

\section{Asymptotic behavior of risks of Bayesian predictive distributions}
\label{section:main_asymptotic}

In this section, we investigate the asymptotic property of the risk of Bayesian predictive distributions based on prior distributions (\ref{definition:relatively_invariant_priors_wishart}).
We show that our prior distribution (\ref{definition:right_invariant_prior}) asymptotically dominates the reference prior distribution (\ref{definition:reference_prior}).

In the next proposition, we show that the family $\{ \pi_{\mathbf{t}} \left( \phi \right) d\phi \}_{\mathbf{t} \in \mathbb{R}^h}$ of prior distributions (\ref{definition:relatively_invariant_priors_wishart}) is related to the Laplace--Beltrami operator.
The Laplace--Beltrami operator is a differential operator that transforms a scalar function defined on the parameter space into another scalar function.
It is independent of the choice of parameterization of the parameter space (see \cite{jost2008riemannian} for the mathematical details).
Appendix \ref{section:LB} provides a precise definition of the Laplace--Beltrami operator.
In a previous study (see \cite{komaki2006shrinkage}), the relationship between the super-harmonicity of the Laplace--Beltrami operator and its dominance over the Jeffreys prior distribution has been discussed.
Based on these approaches, we focus on the eigenfunctions and eigenvalues of the Laplace--Beltrami operator in this study (see also \cite{9429230}).

\begin{proposition} 
  \label{proposition:main_asymptotic}
  The scalar function
  \begin{align}
    K_{ \mathbf{t}} \left( \phi \right)
    :=
    \left( \frac{ \pi_{ \mathbf{t}} \left( \phi \right) }{ \pi_J \left( \phi \right) } \right)^{\frac{1}{2}}
    =
    \prod_{i=1}^h
    \left| \xi^{(i)}_0 \right|^{ \frac{1}{2} \left( t^{(i)} + \frac{r-1}{2} d + 1 \right) }
    \label{definition:K_general}
  \end{align}
  for the prior distribution (\ref{definition:relatively_invariant_priors_wishart}) is an eigenfunction of the Laplace--Beltrami operator with an eigenvalue
  \begin{align}
    \frac{ \LB \, K_{ \mathbf{t} } }{ K_{ \mathbf{t} } }
    =
    \sum_{i=1}^h
    \left(
      \frac{1}{4} k^{(i)}
      \left( t^{(i)} - t^{(i)}_R \right)^2
      -
      \frac{d^2}{16} 
      k^{(i)}
      \left( r - 2 r_{(i)} + k^{(i)} \right)^2
    \right)
    \,.
    \label{formula:eigenvalue}
	\end{align}
\end{proposition}

\begin{proof}[\textbf{\upshape Proof:}]
  See Appendix \ref{section:LB}.
\end{proof}

Our prior distribution $\pi_R \left( \phi \right) d\phi := \pi_{\mathbf{t}_R} \left( \phi \right) d\phi$ has the minimum eigenvalue
\begin{align*}
  \frac{ \LB \, K_{ \mathbf{t}_R } }{ K_{ \mathbf{t}_R } }
  =
  \sum_{i=1}^h
  \left(
    -
    \frac{d^2}{16} 
    k^{(i)}
    \left( r - 2 r_{(i)} + k^{(i)} \right)^2
  \right)
\end{align*}
among the family $\{ \pi_{\mathbf{t}} \left( \phi \right) d\phi \}_{\mathbf{t} \in \mathbb{R}^h}$ of prior distributions.
We asymptotically compare the risk of the Bayesian predictive distribution $\delta^{\mu, \nu}_{\mathbf{t}}$ based on the prior distribution (\ref{definition:relatively_invariant_priors_wishart}) with that of the Bayesian predictive distribution $\delta^{\mu, \nu}_{R}$ based on our prior distributions $\pi_R \left( \phi \right) d\phi$ in the following theorem.
Moreover, our prior distribution $\pi_R \left( \phi \right) d\phi$ asymptotically yields the minimum risk among the family $\{ \pi_{\mathbf{t}} \left( \phi \right) d\phi \}_{\mathbf{t} \in \mathbb{R}^h}$ of prior distributions.

\begin{theorem}
  \label{theorem:main_asymptotic}
  We have
  as $\mu \to \infty$ for each $\phi \in \Phi$
  \begin{align}
    R^{\mu, \nu} \left( \phi, \delta^{\mu, \nu}_{\mathbf{t}} \right)
    &=
    R^{\mu, \nu} \left( \phi, \delta^{\mu, \nu}_{R} \right)
    +
    \frac{\nu}{\mu^2}
    \left(
      \sum_{i=1}^h
      \frac{1}{2} k^{(i)}
      \left( t^{(i)} - t^{(i)}_R \right)^2
    \right)
    +
    O \left( \mu^{-\frac{5}{2}} \right)
    \label{formula:main_asymptotic}
    \,.
  \end{align}
\end{theorem}

\begin{proof}[\textbf{\upshape Proof:}]
  We have
  \begin{align}
    R^{\mu, \nu} \left( \phi, \delta^{\mu, \nu}_{\mathbf{t}} \right)
    =
    R^{\mu, \nu} \left( \phi, \delta^{\mu, \nu}_{J} \right)
    + 2 \frac{ \LB K \left( \phi \right) }{ K \left( \phi \right) } \frac{\nu}{\mu^2}
    + O \left( \mu^{-\frac{5}{2}} \right)
    \label{formula:asymptotic_expansion_R_superharmonicity}
  \end{align}
  as $\mu \to \infty$ for each $\phi \in \Phi$ (see Equation (4) in \cite{komaki2006shrinkage}).
  Therefore, we have
  \begin{align}
    R^{\mu, \nu} \left( \phi, \delta^{\mu, \nu}_{\mathbf{t}} \right)
    &=
    R^{\mu, \nu} \left( \phi, \delta^{\mu, \nu}_{J} \right)
    +
    \frac{\nu}{\mu^2}
    \sum_{i=1}^h
    \left(
      \frac{1}{2} k^{(i)}
      \left( t^{(i)} - t^{(i)}_R \right)^2
      -
      \frac{d^2}{8} 
      k^{(i)}
      \left( r - 2 r_{(i)} + k^{(i)} \right)^2
    \right)
    +
    O \left( \mu^{-\frac{5}{2}} \right)
    \label{formula:main_asymptotic_1}
  \end{align}
  as $\mu \to \infty$ for each $\phi \in \Phi$.
  Particularly, we have
  \begin{align}
    R^{\mu, \nu} \left( \phi, \delta^{\mu, \nu}_{R} \right)
    &=
    R^{\mu, \nu} \left( \phi, \delta^{\mu, \nu}_{J} \right)
    +
    \frac{\nu}{\mu^2}
    \sum_{i=1}^h
    \left(
      -
      \frac{d^2}{8} 
      k^{(i)}
      \left( r - 2 r_{(i)} + k^{(i)} \right)^2
    \right)
    +
    O \left( \mu^{-\frac{5}{2}} \right)
    \label{formula:main_asymptotic_2}
  \end{align}
  for each $\phi \in \Phi$.
  The asymptotic expansions (\ref{formula:main_asymptotic_1}) and (\ref{formula:main_asymptotic_2}) provide the asymptotic expansion (\ref{formula:main_asymptotic}).
\end{proof}

According to Theorem \ref{theorem:main_asymptotic}, the Bayesian predictive distribution $\delta^{\mu, \nu}_R$ based on our prior distribution $\pi_R \left( \phi \right) d\phi$ asymptotically dominates the Bayesian predictive distribution $\delta^{\mu, \nu}_C$ based on the reference prior distribution $\pi_C \left( \phi \right) d\phi$.
This phenomenon motivates us to investigate our prior distribution $\pi_R \left( \phi \right) d\phi$.
Note that the residual $O \left( \mu^{-5/2} \right)$ term in (\ref{formula:main_asymptotic}) may depend on the choice of $\phi \in \Phi$.
Therefore, at this stage, the convergence may not be uniform on the parameter space $\Phi$ as $\mu \to \infty$.
However, this phenomenon is not an asymptotic property as demonstrated in later sections.

\section{Exact behavior of risks of Bayesian predictive distributions}
\label{section:main_exact}

In this section, we explicitly calculate the risks of Bayesian predictive distributions without using asymptotic expansions and clarify its dependency on the sizes of current and future observations.
We prove that the asymptotic property discussed in Theorem \ref{theorem:main_asymptotic} holds for any value of $\mu$ and $\nu$.

The conditional reducibility  of the family of Wishart distributions enables us to decompose the risk of a Bayesian predictive distribution into the parts that are related to the corresponding groups of parameters.
In the next proposition, we show that the risk $R^{\mu, \nu} \left( \phi, \delta^{\mu, \nu}_{\mathbf{t}} \right)$ of the Bayesian predictive distribution $\delta^{\mu, \nu}_{\mathbf{t}}$ based on the prior distribution (\ref{definition:relatively_invariant_priors_wishart}) decomposes into the sum $\sum_{i=1}^h R^{ (i) \mu, \nu} \left( t^{(i)} \right)$ of constant terms, where each constant term $R^{ (i) \mu, \nu} \left( t^{(i)} \right)$ only depends on the prior distribution (\ref{definition:relatively_invariant_priors_wishart_i}) on the parameter space $\Phi^{(i)}$.
That is, the risk $R^{\mu, \nu} \left( \phi, \delta^{\mu, \nu}_{\mathbf{t}} \right)$ is a function $R^{ \mu, \nu} \left( \mathbf{t} \right)$ of the hyperparameter $\mathbf{t} \in \mathbb{R}^h$ and independent of the value of $\phi \in \Phi$.
Each constant term $R^{ (i) \mu, \nu} \left( t^{(i)} \right)$ corresponds to the risk (\ref{definition:R_i}), which is defined as the average of the Kullback--Leibler divergence from the true conditional distribution, of the conditional Bayesian predictive distribution.

\begin{proposition}
  \label{proposition:main_exact}
  Let $\mu > \left( r - 1 \right) d / 2$ and $\nu > \left( r - 1 \right) d / 2$.
  The risk $R^{\mu, \nu} \left( \phi, \delta^{\mu, \nu}_{\mathbf{t}} \right)$ of the Bayesian predictive distribution $\delta^{\mu, \nu}_{\mathbf{t}}$ based on the prior distribution (\ref{definition:relatively_invariant_priors_wishart}) is calculated as
  \begin{align}
    R^{ \mu, \nu} \left( \mathbf{t} \right)
    :=
    \sum_{i=1}^h R^{ (i) \mu, \nu} \left( t^{(i)} \right)
    \,,
    \label{definition:R_t}
  \end{align}
  where
  \begin{align}
    R^{ (i) \mu, \nu} \left( t^{(i)} \right)
    &:=
    - \nu k^{(i)} 
    - \log \Gamma_{k^{(i)}} \left( t^{(i)} + \mu + \nu + \frac{n_{(i)}}{r_{(i)}} \right)
    + \log \Gamma_{k^{(i)}} \left( t^{(i)} + \mu + \frac{n_{(i)}}{r_{(i)}} \right)
    \nonumber\\
    &\quad
    + \left( t^{(i)} + \mu + \nu + \frac{n_{(i)}}{r_{(i)}} \right)
    \psi_{r_{(i)}} \left( \mu + \nu \right)
    - \left( t^{(i)} + \mu + \frac{n_{(i)}}{r_{(i)}} \right)
    \psi_{r_{(i)}} \left( \mu \right)
    \nonumber\\
    &\quad
    - \left( t^{(i)} + \mu + \nu + \frac{n_{(i-1)}}{r_{(i-1)}} \right)
    \psi_{r_{(i-1)}} \left( \mu + \nu \right)
    + \left( t^{(i)} + \mu + \frac{n_{(i-1)}}{r_{(i-1)}} \right)
    \psi_{r_{(i-1)}} \left( \mu \right)
    \label{definitoin:R_i}
  \end{align}
  for $t^{(i)} > - \mu - \left( \left( r_{(i)} - k^{(i)} \right) / 2 \right) d - 1$.
  The last two terms in (\ref{definitoin:R_i}) are interpreted as $0$ if $i = 1$.
\end{proposition}

\begin{proof}[\textbf{\upshape Proof:}]
  See Appendix \ref{section:ESCPD}.
\end{proof}

The asymptotic expansion of (\ref{definitoin:R_i}) is calculated below using the asymptotic expansion (\ref{formula:multivariate_loggamma_mu_x}) of the multivariate log-gamma function and the asymptotic expansion (\ref{formula:multivariate_polygamma_mu_x}) of the multivariate polygamma function in Appendix \ref{section:MPF}.
\begin{align}
  R^{ (i) \mu, \nu} \left( t^{(i)} \right)
  =
  R^{ (i) \mu, \nu} \left( t^{(i)}_R \right)
  +
  \frac{\nu}{\mu^2}
  \left(
    \frac{1}{2} k^{(i)}
    \left( t^{(i)} - t^{(i)}_R \right)^2
  \right)
  +
  O \left( \mu^{-3} \right)
  \label{formula:main_asymptotic_i}
\end{align}
as $\mu \to \infty$,
where
\begin{align}
  R^{ (i) \mu, \nu} \left( t^{(i)}_R \right)
  &=
  \frac{\nu}{\mu}
  \left(
    \frac{1}{2} k^{(i)} + \frac{d}{4} k^{(i)} \left( 2 r_{(i)} - k^{(i)} - 1 \right)
  \right)
  -
  \frac{\nu^2}{\mu^2}
  \left(
    \frac{1}{4} k^{(i)}
    +
    \frac{d}{8} k^{(i)} \left( 2 r_{(i)} - k^{(i)} - 1 \right)
  \right)
  \nonumber\\
  &\quad+
  \frac{\nu}{\mu^2}
  \left(
    \frac{1}{6} k^{(i)}
    + \frac{d}{4} k^{(i)} \left( 2 r_{(i)} - k^{(i)} - 1 \right)
    + \frac{d^2}{24}
    k^{(i)}
    \left(
      3 {r_{(i)}}^2
      - 6 r_{(i)}
      - 3 k^{(i)} \left( r_{(i)} - 1 \right)
      + 2 {k^{(i)}}^2
      + 1
    \right)
  \right)
  +
  O \left( \mu^{-3} \right)
  \nonumber\\
  &=
  \frac{1}{2} \frac{\nu}{\mu} \left( n_{(i)} - n_{(i-1)} \right)
  + O \left( \mu^{-2} \right)
  \label{formula:main_asymptotic_Ri}
\end{align}
is the asymptotic expansion of the risk $R^{ (i) \mu, \nu} \left( t^{(i)}_R \right)$ as $\mu \to \infty$.
The asymptotic expansion (\ref{formula:main_asymptotic_i}) is consistent with the asymptotic property discussed in Theorem \ref{theorem:main_asymptotic}.
The expansion  shows that the term (\ref{definitoin:R_i}) asymptotically achieves its minimum at $t^{(i)} = t^{(i)}_R$.

This phenomenon is not an asymptotic property.
In the next theorem, we show that the term (\ref{definitoin:R_i}) takes its unique minimum at $t^{(i)} = t^{(i)}_R$ for any value of $\mu$ and $\nu$.
In other words, the function $R^{ \mu, \nu} \left( \mathbf{t} \right)$ of the hyperparameter $\mathbf{t} \in \mathbb{R}^h$ has its unique minimum at $\mathbf{t} = \mathbf{t}_R$ for any value of $\mu$ and $\nu$.

\begin{theorem}
  \label{theorem:main_exact}
  Let $\mu > \left( r - 1 \right) d / 2$ and $\nu > \left( r - 1 \right) d / 2$.
  The risk $R^{\mu, \nu} \left( \mathbf{t} \right)$ is a convex function of $\mathbf{t}$ and attains its minimum $R^{\mu, \nu}_R := R^{\mu, \nu} \left( \mathbf{t}_R \right)$ at $\mathbf{t} = \mathbf{t}_R$.
\end{theorem}

\begin{proof}[\textbf{\upshape Proof:}]
  The derivatives of the risk $R^{ (i) \mu, \nu} \left( t^{(i)} \right)$ with respect to $t^{(i)}$ are calculated as
  \begin{align*}
    \frac{\partial}{\partial t^{(i)}}
    R^{ (i) \mu, \nu} \left( t^{(i)} \right)
    &=
    -
    \psi_{k^{(i)}} \left( t^{(i)} + \mu + \nu + \frac{n_{(i)}}{r_{(i)}} \right)
    +
    \psi_{k^{(i)}} \left( t^{(i)} + \mu + \frac{n_{(i)}}{r_{(i)}} \right)
    +
    \psi_{k^{(i)}} \left( \mu + \nu - \frac{ r_{(i)} - k^{(i)} }{2} d \right)
    -
    \psi_{k^{(i)}} \left( \mu - \frac{ r_{(i)} - k^{(i)} }{2} d \right)
    \,,\\
    \frac{\partial^2}{\partial {t^{(i)}}^2}
    R^{ (i) \mu, \nu} \left( t^{(i)} \right)
    &=
    - \psi^{(1)}_{k^{(i)}} \left( t^{(i)} + \mu + \nu + \frac{n_{(i)}}{r_{(i)}}\right)
    + \psi^{(1)}_{k^{(i)}} \left( t^{(i)} + \mu + \frac{n_{(i)}}{r_{(i)}} \right)
    \,,
  \end{align*}
  where we employ (\ref{formula:polygamma}).
  Each $R^{ (i) \mu, \nu} \left( t^{(i)} \right)$ takes its minimum at $t^{(i)} = t^{(i)}_R$
  because
  \begin{align*}
  \left. \frac{\partial}{\partial t^{(i)}} \right|_{t^{(i)} = t^{(i)}_R } R^{ (i) \mu, \nu} \left( t^{(i)} \right) = 0
  \,,\;
  \frac{\partial^2}{\partial {t^{(i)}}^2} R^{ (i) \mu, \nu} \left( t^{(i)} \right) > 0
  \,.
  \end{align*}
\end{proof}

Theorem \ref{theorem:main_exact} shows that the Bayesian predictive distribution $\delta^{\mu, \nu}_R$ based on our prior distribution $\pi_R \left( \phi \right) d\phi$ dominates the Bayesian predictive distribution $\delta^{\mu, \nu}_C$ based on the reference prior distribution $\pi_C \left( \phi \right) d\phi$ for any value of $\mu$ and $\nu$.

We define the normalized risk
\begin{align}
  \mathrm{NR}^{\mu, \nu} \left( \mathbf{t} \right)
  :=
  \frac{\mu}{\nu} R^{\mu, \nu} \left( t \right)
  \label{definition:NR}
\end{align}
of the risk (\ref{definition:R_t}).
The asymptotic expansions (\ref{formula:main_asymptotic_i}) and (\ref{formula:main_asymptotic_Ri}) show that
\begin{align}
  \mathrm{NR}^{\mu, \nu} \left( \mathbf{t} \right) = \frac{n}{2} + O \left( \mu^{-1} \right)
  \label{formula:NR}
\end{align}
as $\mu \to \infty$, where the residual $O \left( \mu^{-1} \right)$ term may depend on the value of $\nu$.
The first $O \left( 1 \right)$ term $n / 2$ in (\ref{formula:NR}) is the expected term because we used $n$ real parameters $\xi \in \Xi$ for the prediction of $Y \sim W_r \left( \nu, \xi \right)$ (see \cite{komaki2009bayesian} for example).
Theorem \ref{theorem:main_asymptotic} shows how the choice of the prior distribution (\ref{definition:relatively_invariant_priors_wishart}) asymptotically affects the performance of the prediction of $Y \sim W_r \left( \nu, \xi \right)$ as $\mu \to \infty$.
Recall that, in the case of the prediction of the sample variance-covariance matrix of the multivariate normal distributions, a large value of $\mu$ corresponds to a large sample size of the observation.
However, contrary to expectations, the expression (\ref{formula:main_asymptotic}) is silent regarding the behavior of the performance of the prediction of $Y \sim W_r \left( \nu, \xi \right)$ when this sample $\mu$ is small.
Theorem \ref{theorem:main_exact} ensures that the condition for the relative order of the risks (\ref{definition:R_t}) is preserved when the size of the observation is relatively small.

Lastly, let us remark on the parameter $\nu > \left( r - 1 \right) d / 2$.
This parameter $\nu$ corresponds to the size of the prediction.
Therefore, we might expect that the risk (\ref{definition:R}) should be proportional to this parameter $\nu$.
However, the risk (\ref{definition:R}) of the Bayesian predictive distribution is not linear in the parameter $\nu$, exhibiting an interesting property of the Bayesian prediction.

\section{Geometry of Bayesian predictive distributions}
\label{section:geometry_of_bayesian_predictive_distribution}

In the previous section, we showed that the risk of the Bayesian predictive distribution based on the prior distribution (\ref{definition:relatively_invariant_priors_wishart}) decomposes into a sum of constant terms and depends only on the value of the hyperparameter $\mathbf{t} \in \mathbb{R}^h$.
In this section, we consider case $h = 2$ and investigate the dependence of the performance of the prior distribution (\ref{definition:relatively_invariant_priors_wishart}) on the values of $\mu$ and $\nu$.
We set the partition $\mathbf{k} = \left( k^{(1)}, k^{(2)} \right) = \left( k, r - k \right)$. Here, we assume $r \geq 2$ and $0 < k < r$.
In the context of the reference approach, $\phi^{(1)} = \left( \zeta_1 \right) \in \Phi^{(1)}$ and $\phi^{(2)} = \left( \xi_{1/2}, \xi_0 \right) \in \Phi^{(2)}$ correspond to the parameter of interest and nuisance parameter, respectively.

First, we interpret Theorem \ref{proposition:main_exact} for case $h = 2$.
Recall that we have the decomposition $R^{ \mu, \nu} \left( \mathbf{t} \right) = R^{ (1) \mu, \nu} \left( t^{(1)} \right) + R^{ (2) \mu, \nu} \left( t^{(2)} \right)$ for $\mathbf{t} = \left( t^{(1)}, t^{(2)} \right)$.
The term $R^{ (1) \mu, \nu} \left( t^{(1)} \right)$ corresponds to the risk of the Bayesian predictive distribution $\delta^{(1) \mu, \nu} \left( y^{(1)} \;\middle|\; x^{(1)} \right) dy^{(1)}$ for the random variable $Y^{(1)}$ and the term $R^{ (2) \mu, \nu} \left( t^{(2)} \right)$ corresponds to the risk of the conditional Bayesian predictive distribution $\delta^{(2) \mu, \nu} \left( y^{(2)} \;\middle|\; x^{(1)}, x^{(2)}, y^{(1)} \right) dy^{(2)}$ for the conditional random variable $\left. Y^{(2)} \;\middle|\; Y^{(1)} \right.$.
In the next proposition, we show the relationship of the magnitude between the risks $R^{\mu, \nu}_R$, $R^{\mu, \nu}_C$, and $R^{\mu, \nu}_J$.

\begin{proposition}
  \label{proposition:main_exact_h2}
  We obtain
  \begin{align*}
    R^{(1) \mu, \nu}_R = R^{(1) \mu,  \nu}_C < R^{(1) \mu, \nu}_J
    \,,\;
    R^{(2) \mu, \nu}_R < R^{(2) \mu,  \nu}_C = R^{(2) \mu, \nu}_J
    \,.
  \end{align*}
  Therefore,
  \begin{align*}
    R^{\mu, \nu}_R < R^{\mu, \nu}_C < R^{\mu, \nu}_J
  \end{align*}
  for any $\mu > \left( r - 1 \right) d / 2$ and $\nu > \left( r - 1 \right) d / 2$.
\end{proposition}

\begin{proof}[\textbf{\upshape Proof:}]
Because $r \geq 2$ and $0 < k < r$, we have $t^{(1)}_J < t^{(1)}_C = t^{(1)}_R$ and $t^{(2)}_R < t^{(2)}_C = t^{(2)}_J$.
Each $R^{ (i) \mu, \nu} \left( t^{(i)} \right)$ takes its minimum at $t^{(i)} = t^{(i)}_R$ (see Theorem \ref{theorem:main_exact}).
\end{proof}

Proposition \ref{proposition:main_exact_h2} shows that the Bayesian predictive distribution $\delta^{\mu, \nu}_R$ based on our prior distribution $\pi_R \left( \phi \right) d\phi$ dominates the Bayesian predictive distribution $\delta^{\mu, \nu}_C$ based on the reference prior distribution $\pi_C \left( \phi \right) d\phi$.
This is because the use of prior distribution $\pi_R \left( \phi \right) d\phi$ improves the performance on the space $\Phi^{(2)}$ of the nuisance parameter $\phi^{(2)}$ compared to the use of the reference prior distribution $\pi_C \left( \phi \right) d\phi$.

Subsequently, we consider the difference in risks $R^{\mu, \nu} \left( \mathbf{t} \right)$ and $R^{\mu, \nu}_{J} = R^{\mu, \nu} \left( \mathbf{t}_J \right)$.
The subset $T^{\mu}$ is defined as
\begin{align*}
  T^{\mu}
  :=
  \left\{
    \mathbf{t} = \left( t^{(1)}, t^{(1)} \right) \in T = \mathbb{R}^2
  \;\middle|\;
    t^{(1)} > - \mu - 1 \;,\; t^{(2)} >  - \mu - \frac{k}{2}d - 1
  \right\}
\end{align*}
of the set $T := \mathbb{R}^2$ for $\mu > \left( r - 1 \right) d / 2$.
If $\mathbf{t} \in T^{\mu}$, then the posterior distribution $\pi^{\mu}_{\mathbf{t}} \left( \phi \mid x \right) d\phi$ of the prior distribution (\ref{definition:relatively_invariant_priors_wishart}) given an observation $x$ of $X \sim W \left( \mu, \phi \right)$ is proper.
The normalized risk difference $\mathrm{NRD}^{\mu, \nu} \left( \mathbf{t} \right)$ is defined as
\begin{align}
  \mathrm{NRD}^{\mu, \nu} \left( \mathbf{t} \right)
  := \frac{\mu^2}{\nu} \left( R^{\mu, \nu} \left( \mathbf{t} \right) - R^{\mu, \nu}_{J} \right)
  \label{definition:normalized_risk_difference}
\end{align}
for $\mathbf{t} \in T^{\mu}$.
The subset $D^{\mu, \nu}$ of the set $T$ is given by
\begin{align*}
  D^{\mu, \nu}
  :=
  \left\{
    \mathbf{t} \in T^{\mu}  
  \;\middle|\;
    \mathrm{NRD}^{\mu, \nu} \left( \mathbf{t} \right) < 0
  \right\}
\end{align*}
for $\mu > \left( r - 1 \right) d / 2$ and $\nu > \left( r - 1 \right) d / 2$.
The risk $R^{\mu, \nu} \left( \mathbf{t} \right)$ is a strictly convex function of $\mathbf{t}$ on $T^{\mu}$;
thus, the subset $D^{\mu, \nu}$ is a convex set in $T$.

First, consider the behaviors of the normalized risk difference $\mathrm{NRD}^{\mu, \nu} \left( \mathbf{t} \right)$ and the subset $D^{\mu, \nu}$ of $T$ when the size $\mu$ of the observation is large.
Based on Theorem \ref{theorem:main_asymptotic}, the asymptotic expansion is given as
\begin{align}
  \mathrm{NRD}^{\mu, \nu} \left( t \right)
  &=
  \frac{k}{2} \left( t^{(1)} - t^{(1)}_R \right)^2 + \frac{r-k}{2} \left( t^{(2)} - t^{(2)}_R \right)^2
  - \left( R^{\mu, \nu}_J - R^{\mu, \nu}_R \right)
  + O \left( \mu^{-1} \right)
  \nonumber\\
  &=
  \frac{k}{2} \left( t^{(1)} + \left( \frac{ k - 1 }{2} d + 1 \right) \right)^2 + \frac{r-k}{2} \left( t^{(2)} + \left( \frac{ r + k - 1 }{2} d + 1 \right) \right)^2
  - \frac{d^2}{8} r k \left( r - k \right)
  + O \left( \mu^{-1} \right)
  \label{formula:NRD_asymptotic_expansion_infinity}
\end{align}
of the normalized risk difference (\ref{definition:normalized_risk_difference}) and the limit
\begin{align*}
  D^{\infty}
  :=
  \lim_{\mu \to \infty} D^{\mu, \nu}
  =
  \left\{
    \mathbf{t} \in T
  \;\middle|\;
    \frac{k}{2} \left( t^{(1)} - t^{(1)}_R \right)^2 + \frac{r-k}{2} \left( t^{(2)} - t^{(2)}_R \right)^2
    - \frac{d^2}{8} r k \left( r - k \right) < 0
  \right\}
\end{align*}
of the sequence of subsets $D^{\mu, \nu}$ of $T$ as $\mu \to \infty$.
Note that the subset $D^{\infty}$ of $T$ does not depend on the value of $\nu$.
The region $D^{\infty}$ is an oval in plane $T$ centered at point $\mathbf{t}_R$.
Point $\mathbf{t}_J$ is located at the boundary $\partial D^{\infty}$ of set $D^{\infty}$,
and point $\mathbf{t}_C$ is located in set $D^{\infty}$.

Next, consider the case when the size $\mu$ of the observation is small.
We have
\begin{align*}
  T^{\frac{r-1}{2}d}
  :=
  \lim_{\mu \to \frac{r-1}{2}d + 0} T^{\mu}
  =
  \left\{
    \mathbf{t} = \left( t^{(1)}, t^{(2)} \right) \in T = \mathbb{R}^2
  \;\middle|\;
    t^{(1)} > t_J^{(1)}
    \;,\;
    t^{(2)} > t_R^{(2)}
  \right\}
\end{align*}
and
\begin{align*}
  \lim_{\mu \to \frac{r-1}{2}d + 0} \mathrm{NRD} \left( t \right)
  =
  \begin{cases}
    + \infty &\quad\text{if $t^{(2)} \geq t^{(2)}_J$} \\
    - \infty &\quad\text{if $t^{(2)} < t^{(2)}_J$}
  \end{cases}
\end{align*}
for $t \in T^{\frac{r-1}{2}d}$.
The limit is set to
\begin{align*}
  D^{\frac{r-1}{2}d}
  :=
  \lim_{\mu \to \frac{r-1}{2}d + 0} D^{\mu, \nu}
  = \left\{
    \mathbf{t} \in T
  \;\middle|\;
    t^{(1)}_J < t^{(1)}
    \;,\;
    t^{(1)}_R < t^{(2)} < t^{(1)}_J
    \right\}
\end{align*}
of the sequence of the subsets $D^{\mu, \nu}$ of $T$ as $\mu \to \left( r - 1 \right) d / 2 + 0$.
Note that the subset $D^{\frac{r-1}{2}d}$ of $T$ does not depend on the value of $\nu$.
The three points $\mathbf{t}_J$, $\mathbf{t}_C$, and $\mathbf{t}_R$ are located at the boundary $\partial D^{\frac{r-1}{2}d}$ of the set $D^{\frac{r-1}{2}d}$.
Here, the region $D^{\frac{r-1}{2}d}$ is very different from the region $D^{\infty}$.

The convex region $D^{\mu, \nu}$ in $T$ gradually changes its shape from the rectangle $D^{\frac{r-1}{2}d}$ to the oval $D^{\infty}$ as the value of $\mu$ increases from $\left( r - 1 \right) d / 2$ to $+\infty$,
indicating that its dominance over the Jeffreys prior distribution depends on the value of $\mu$.
The asymptotic expansion (\ref{formula:NRD_asymptotic_expansion_infinity}) does not necessarily explain the behavior of the risk when the value of $\mu$ is small.
The set $D^{\mu, \nu}$ contains all the points $\mathbf{t} = \left( t^{(1)}, t^{(2)} \right)$ such that $t^{(1)}_J < t^{(1)} < t^{(1)}_R$ and $t^{(2)}_R < t^{(2)} < t^{(2)}_J$.
\figurename \ref{figure:T} shows the placement of the three points $\mathbf{t}_J$, $\mathbf{t}_C$, and $\mathbf{t}_R$ and the regions $D^{\frac{r-1}{2}d}$ and $D^{\infty}$ for $(d, r, k) = (1, 2, 1)$.
\begin{figure}[htbp]
  \centering
  \includegraphics[width=0.75\columnwidth]{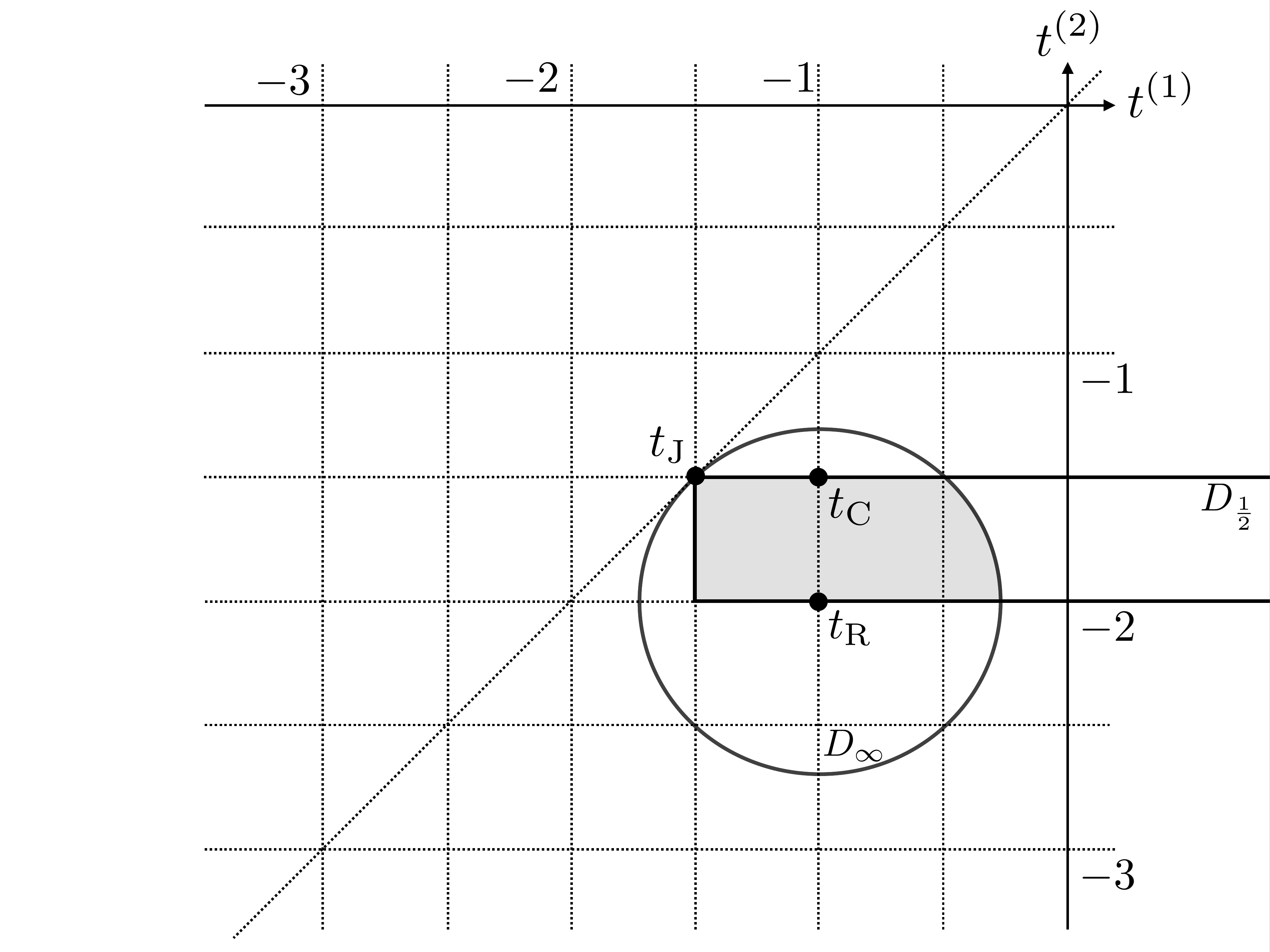}
  \caption{The placement of $\mathbf{t}_J$, $\mathbf{t}_C$ $\mathbf{t}_R$, $D^{\frac{r-1}{2}d}$, and $D^{\infty}$ for $(d, r, k) = (1, 2, 1)$ in the space $T = \mathbb{R}^2$. The region $D^{\frac{1}{2}} \cap D^{\infty}$ is shaded.}
  \label{figure:T}
\end{figure}

The subset $V^{\nu}$ is defined as
\begin{align*}
  V^{\nu} := \bigcap_{\mu > \frac{r-1}{2}d } D^{\mu, \nu}
\end{align*}
of set $T$ for each $\nu > \left( r - 1 \right) d / 2$.
Each set $D^{\mu, \nu}$ is convex in $T$;
therefore, the set $V^{\nu}$ is also convex in $T$.
If $\mathbf{t} \in V^{\nu}$, then the corresponding prior distribution $\pi_{\mathbf{t}} \left( \phi \right) d\phi$ dominates the Jeffreys prior distribution for any $\mu > \left( r - 1 \right) d / 2$.
We have $V^{\nu} \subset D^{\frac{r-1}{2}d} \cap D^{\infty}$ for $\nu > \left( r - 1 \right) d / 2$.
We suspect that the equality $V^{\nu} = D^{\frac{r-1}{2}d} \cap D^{\infty}$ holds for any $\nu > \left( r - 1 \right) d / 2$.
However, we cannot mathematically prove this equality.
If this equality holds, then the set $V^{\nu}$ does not depend on the value of $\nu$.
The region $D^{\frac{r-1}{2}d} \cap D^{\infty}$ for case $(d, r, k) = (1, 2, 1)$ is shaded in \figurename \ref{figure:T}.

We consider the case $(d, r, k) = (1, 2, 1)$.
Recall that the normalized risk $\mathrm{NR}^{\mu, \nu} \left( \mathbf{t} \right)$ is defined as (\ref{definition:NR}) and is $3/2 + O \left( \mu^{-1} \right)$.
Solid lines in \figurename \ref{figure:t_nu100_exact} show the theoretically expected values of normalized risks for $\nu = 1$,
while dashed lines present the lower terms up to $O \left( \mu^{-1} \right)$.
Recall that these values are constant independent of the value of the true parameter $\xi \in S^{+}_2 \left( \mathbb{R} \right)$.
We observe that the $O \left( \mu^{-1} \right)$ terms of the normalized risks do not necessarily explain the behaviors of the risks when the value of $\mu$ is small.
\figurename \ref{figure:t_nu100} shows the results of numerical simulations concerning the values of the normalized risks for the randomly generated value of
\begin{align*}
  \xi =
  \begin{bmatrix}
    3.583614 & 2.408764 \\
    2.408764 & 4.671542
  \end{bmatrix}
  \in S^{+}_2 \left( \mathbb{R} \right)
\end{align*}
for $(d, r, k) = (1, 2, 1)$ and $\nu = 1$.
Monte Carlo simulation was used to evaluate the expectation (\ref{definition:R}) over $X \sim W_2 \left( \mu, \xi \right)$ and $Y \sim W_2 \left( \nu, \xi \right)$.
\figurename \ref{figure:d} shows the normalized risks (\ref{definition:NR}) of $\mathbf{t} \in T^{\mu}$ for $\mu \in \{ 1/2 + 10^{-3} \,,\, 3/4 \,,\, 1 \,,\, 100 \}$ and $\nu = 1$ for the family $\{ W_2 \left( \mu, \xi \right) \}$ of real Wishart distributions; compare them with \figurename \ref{figure:T}.
The region $D^{\mu, \nu}$, whose normalized risk $\mathrm{NR}^{\mu, \nu} \left( \mathbf{t} \right)$ is less than that of $\mathbf{t} = \mathbf{t}_J$, is colored in blue.
From the experiments, the equality $V^{\nu} = D^{\frac{r-1}{2}d} \cap D^{\infty}$ appears to hold for $(d, r, k) = (1, 2, 1)$ and $\nu = 1$.

\begin{figure}[htbp]
  \begin{tabular}{lr}
    \begin{minipage}[t]{0.475\hsize}
      \centering
      \includegraphics[width=1.0\columnwidth]{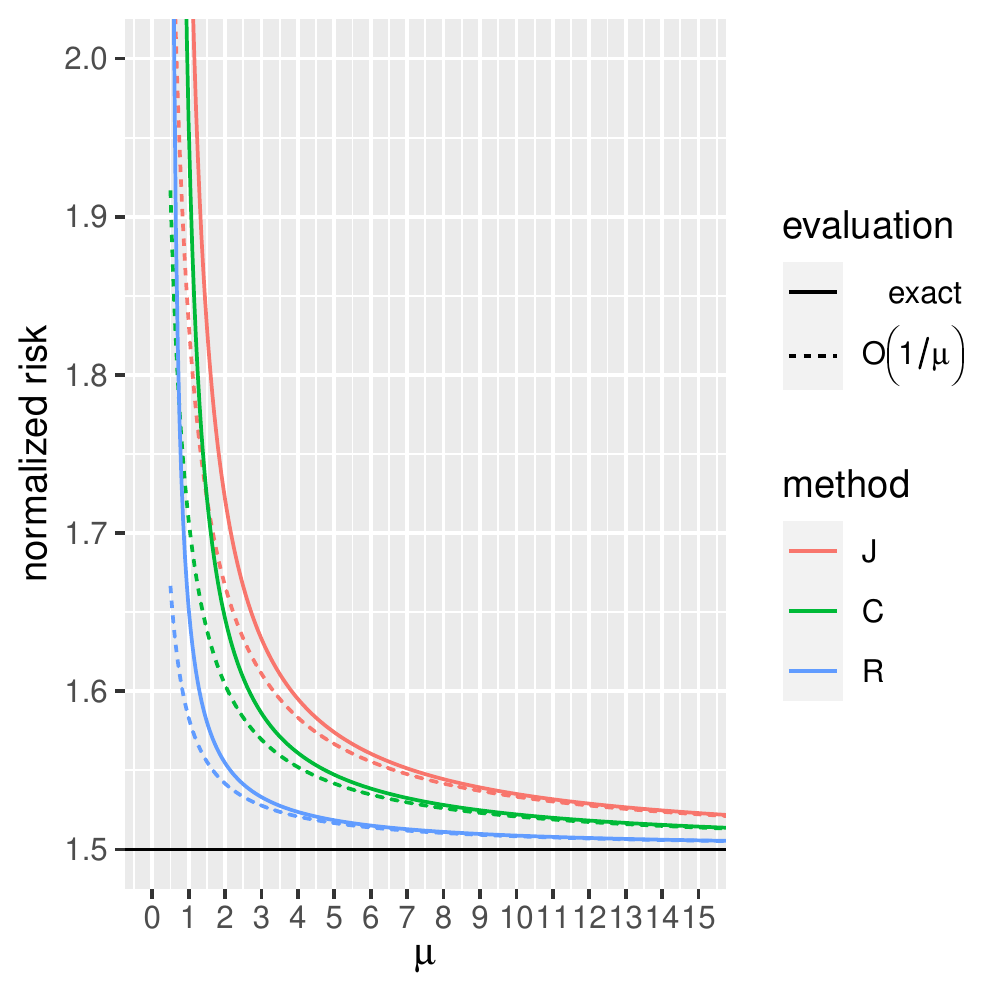}
      \subcaption{Theoretically expected values of normalized risks. Dashed lines show the terms up to $O \left( \mu^{-1} \right)$ of the exact evaluations.}
      \label{figure:t_nu100_exact}
    \end{minipage} &
    \begin{minipage}[t]{0.475\hsize}
      \centering
      \includegraphics[width=1.0\columnwidth]{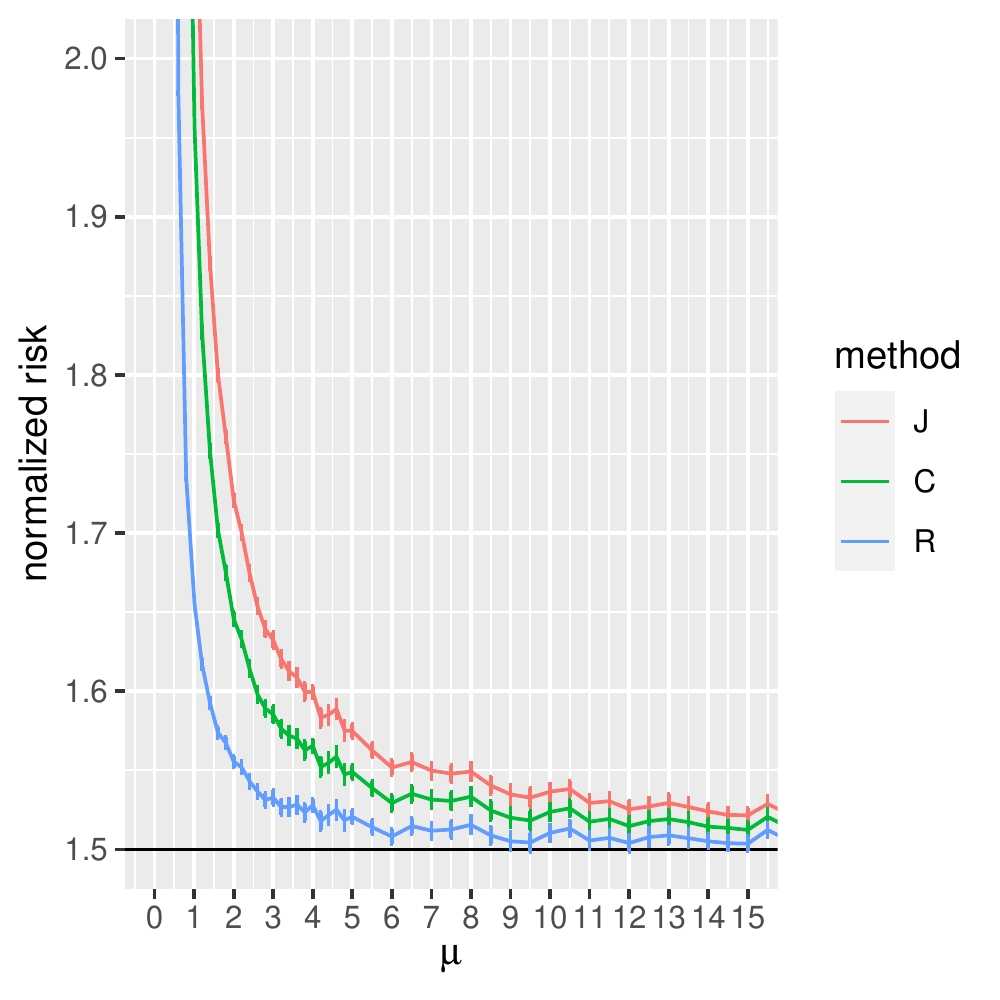}
      \subcaption{Numerical experiments concerning the values of normalized risks for a specific value of $\xi$.}
      \label{figure:t_nu100}
      \end{minipage}
  \end{tabular}
  \caption{The comparison between normalized risks $\mathrm{NR}^{\mu,\nu} \left( \mathbf{t}_J \right)$, $\mathrm{NR}^{\mu,\nu} \left( \mathbf{t}_C \right)$, and $\mathrm{NR}^{\mu,\nu} \left( \mathbf{t}_R \right)$ for $(d, r, k) = (1, 2, 1)$ and $\nu = 1$. We see that these normalized risks are $3/2 + O \left( \mu^{-1} \right)$. We also see that $\mathrm{NR}^{\mu,\nu} \left( \mathbf{t}_R \right) < \mathrm{NR}^{\mu,\nu} \left( \mathbf{t}_C \right) < \mathrm{NR}^{\mu,\nu} \left( \mathbf{t}_J \right)$ for $\mu > 1/2$ (see Proposition \ref{proposition:main_exact_h2}).}
  \label{figure:c}
\end{figure}

\begin{figure}[h]
  \begin{tabular}{lr}
    \begin{minipage}[t]{0.475\hsize}
      \centering
      \includegraphics[width=1.0\columnwidth]{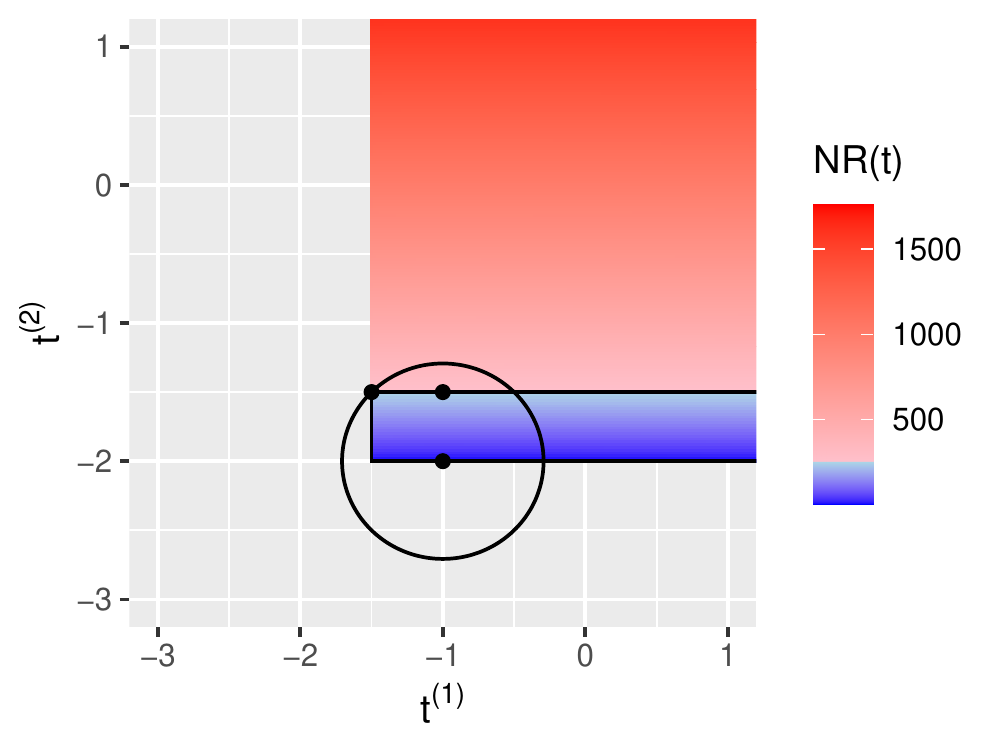}
      \subcaption{$(\mu, \nu) = (1/2 + 10^{-3}, 1)$}
      \label{figure:d_mu00050_n100}
    \end{minipage}
    &
    \begin{minipage}[t]{0.475\hsize}
      \centering
      \includegraphics[width=1.0\columnwidth]{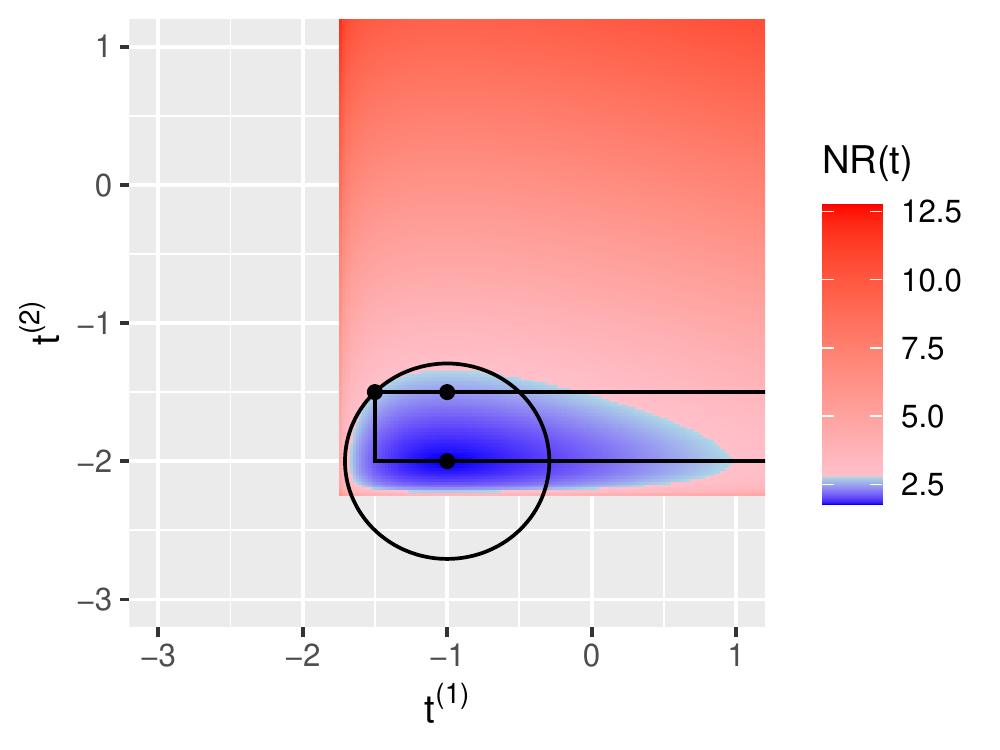}
      \subcaption{$(\mu, \nu) = (3/4, 1)$}
      \label{figure:d_mu00075_n100}
    \end{minipage}
    \\
    \\
    \begin{minipage}[t]{0.475\hsize}
      \centering     
      \includegraphics[width=1.0\columnwidth]{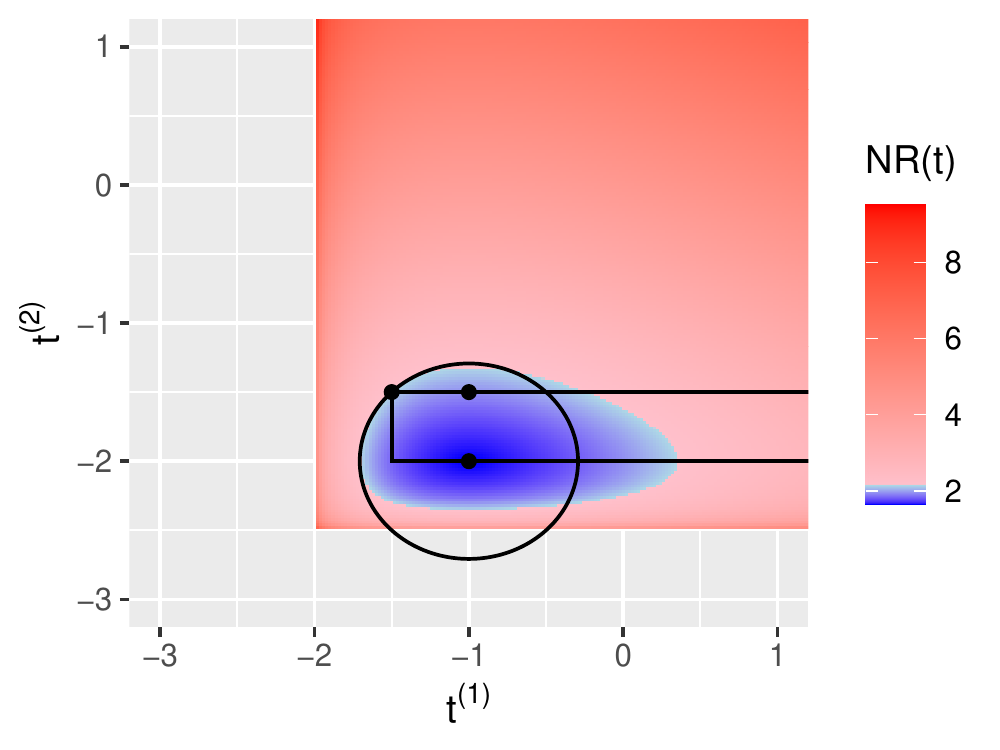}
      \subcaption{$(\mu, \nu) = (1, 1)$}
      \label{figure:d_mu00100_n100}
    \end{minipage}
    &
    \begin{minipage}[t]{0.475\hsize}
      \centering     
      \includegraphics[width=1.0\columnwidth]{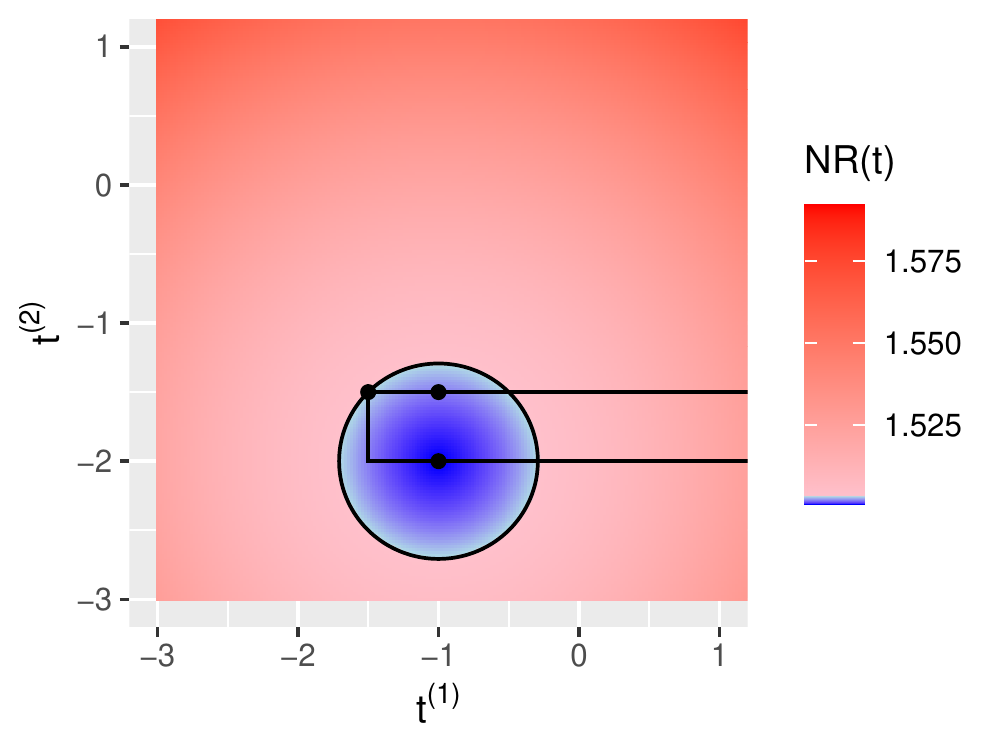}
      \subcaption{$(\mu, \nu) = (100, 1)$}
      \label{figure:d_mu10000_n100}
    \end{minipage}
  \end{tabular}
  \caption{The magnitude of normalized risks $\mathrm{NR}^{\mu,\nu} \left( \mathbf{t} \right)$ for $(d, r, k) = (1, 2, 1)$ and $\nu = 1$. The convex region $D^{\mu, \nu}$, which is colored in blue, gradually changes its shape from the rectangle $D^{\frac{1}{2}}$ to the oval $D^{\infty}$ as the value of $\mu$ increases from $1/2$ to $+\infty$. The equality $V = D^{\frac{1}{2}} \cap D^{\infty}$ appears to hold.}
  \label{figure:d}
\end{figure}

\section{Relative invariance under upper-triangular block matrices}
\label{section:relative_invariance}

Another interpretation of Theorem \ref{theorem:main_exact} is given.
The relative invariance under the left action of a group on a parameter space is of great importance in statistics (see \cite{hartigan1964invariant, eaton1983multivariate} for example).
We show that prior distributions (\ref{definition:relatively_invariant_priors_wishart}) are relatively invariant under the left action of the group of upper-triangular block matrices.
We define the group $G_{\mathbf{k}}$ of upper-triangular block matrices as the subgroup
\begin{align*}
  G_{\mathbf{k}}
  :=
  \left\{
    g \in G_r
    \;\middle|\;
    \text{
      $g$ is of the form $\begin{bmatrix} g^{(1)} & * & * \\ 0 & \ddots & * \\0 & 0 & g^{(h)} \end{bmatrix}$, $g^{(k)} \in G_{k^{(i)}}$.
    }
  \right\}
\end{align*}
of the general linear group $G_r$ of rank $r$ for $\mathbf{k} = \left( k^{(1)}, \ldots, k^{(h)} \right)$ with $r = \sum_{i=1}^h k^{(i)}$.
The group $G_{\mathbf{k}}$ defines the left action on the parameter space $\Xi \cong \Phi$ as $g \cdot \xi := g \, \xi \, g^*$ for $g \in G_{\mathbf{k}}$ and $\xi \in \Xi$.
The parameter $\xi^{(i)}_0$ in the parameter space $\Phi^{(i)}$ transforms as $\xi^{(i)}_0 \mapsto g^{(i)} \, \xi^{(i)}_0 \, { g^{(i)} }^*$ for each $i \in \{ 1, \ldots, h \}$.
The determinant of the Jacobian of this transformation in the space $E^+_{k^{(i)}}$ is calculated as $\left| g^{(i)} \right|^{2 \left( \left( k^{(i)} - 1 \right) d / 2 + 1 \right)}$
(see Proposition 5.11 in \cite{eaton1983multivariate} for the case $d=1$, and Proposition III.4.2 in \cite{faraut1994analysis} for the general case).
The measure (\ref{definition:relatively_invariant_priors_wishart_i}) transforms as
\begin{align*}
  \left| \xi_0^{(i)} \right|^{t^{(i)}} d\xi^{(i)}_{1/2} d\xi^{(i)}_0
  &\mapsto
  \left| g^{(i)} \right|^{ 2 t^{(i)} } \left| \xi_0^{(i)} \right|^{t^{(i)}}
  \times
  \left( \prod_{j=1}^{i-1} \left|  g^{(j)} \right| \right)^{ k^{(i)} d } \left| g^{(i)} \right|^{ r_{(i-1)} d } d\xi^{(i)}_{1/2}
  \times
  \left| g^{(i)} \right|^{ 2 \left( \frac{ \left( k^{(i)} - 1 \right) d }{ 2 } + 1 \right) } d\xi^{(i)}_0
\end{align*}
for each $i \in \{ 1, \ldots, h \}$, where $r_{(-1)}$ is interpreted as zero.
Therefore, the measure (\ref{definition:relatively_invariant_priors_wishart}) is transformed as $\pi_{\mathbf{t}} \left( \phi \right) d\phi \mapsto \chi_{\mathbf{t}} \left( g \right) \times \pi_{\mathbf{t}} \left( \phi \right) d\phi$,
where the multiplier is calculated as
\begin{align*}
  \chi_{\mathbf{t}} \left( g \right)
  :=
  \prod_{i=1}^h
  \left| g^{(i)} \right|^{ 2 \left( t^{(i)} - t^{(i)}_J \right) }
  =
  \prod_{i=1}^h
  \left| g^{(i)} \right|^{ 2 \left( t^{(i)} + \frac{ n }{ r } \right) }
\end{align*}
for $g \in G_{\mathbf{k}}$.
Therefore, prior distributions (\ref{definition:relatively_invariant_priors_wishart}) are relatively invariant under the left action of group $G_{\mathbf{k}}$ of the upper-triangular block matrices.

The fact that the prior distribution (\ref{definition:relatively_invariant_priors_wishart}) is relatively invariant under the left action of group $G_{\mathbf{k}}$ is consistent with Proposition \ref{proposition:main_exact}.
This is because the risk of the Bayesian predictive distribution based on a relatively invariant prior distribution is constant (see Theorem 9.3.8 in \cite{robert2007bayesian}).

Among the relatively invariant prior distributions, the left and right invariant prior distributions are important (see \cite{eaton1983multivariate}).
The direct computation shows that the left invariant prior distribution is the Jeffreys prior distribution $\pi_J \left( \phi \right) d\phi$, and the right invariant prior distribution is our prior distribution $\pi_R \left( \phi \right) d\phi$ (see Example 6.14 in \cite{eaton1983multivariate} for the case $d=1$).
The fact that the prior distribution $\pi_R \left( \phi \right) d\phi$ is the right invariant prior distribution under the left action of group $G_{\mathbf{k}}$ is consistent with Theorem \ref{theorem:main_exact}.
This is because the risk of the Bayesian predictive distribution based on the right invariant prior distribution is the minimum among the risks of Bayesian predictive distributions based on relatively invariant prior distributions (see Theorem 9.4.4 in \cite{robert2007bayesian}).

We discuss the minimaxity and admissibility of the Bayesian predictive distribution $\delta^{\mu, \nu}_R$.
First, we consider the minimaxity.
The group $G_{\mathbf{k}}$ is known to be amenable if $h = r$ (see \cite{bondar1981amenability}).
Thus, the Bayesian predictive distribution $\delta^{\mu, \nu}_R$ based on our prior distribution $\pi_R \left( \phi \right) d\phi$ is minimized if $h = r$ (i.e., $\mathbf{k} = \left( 1, \ldots, 1 \right)$), by the Hunt-Stein theorem (see Theorem 9.5.5 in \cite{robert2007bayesian}).
We observe that the finer the partition $\mathbf{k} = \left( k^{(1)}, \ldots, k^{(h)} \right)$ of rank $r$, the lower the risk $R^{\mu, \nu}_R$ of the Bayesian predictive distribution $\delta^{\mu, \nu}_R$ based on our prior distribution $\pi_R \left( \phi \right) d\phi$ with respect to the partition $\mathbf{k}$.
Therefore, the Bayesian predictive distribution $\delta^{\mu, \nu}_R$ based on our prior distribution $\pi_R \left( \phi \right) d\phi$ is minimax if and only if $h = r$.
Subsequently, we consider the admissibility.
If $r  =2$, there exists a prior distribution that dominates the prior distribution $\pi_R \left( \phi \right) d\phi$ (see Theorem 4 in \cite{komaki2009bayesian}).
Therefore, the Bayesian predictive distribution $\delta^{\mu, \nu}_R$ based on our prior distribution $\pi_R \left( \phi \right) d\phi$ is not admissible if $r \geq 2$.
This is because we may consider the partition $\mathbf{k} := \left( 2, 1, \ldots, 1 \right)$ of rank $r$ if $r \geq 2$, that is, $h = r-1$.

\appendix
\section*{Appendices}
\def\thesection{\Alph{section}}

\section{Laplace--Beltrami operator}
\label{section:LB}

We define the Laplace--Beltrami operator for the family $\{ W_r \left( \mu, \phi \right) \}_{\phi \in \Phi}$ of Wishart distributions.
In addition, we also prove Proposition \ref{proposition:main_asymptotic}.

The space $S_r \left( \mathbb{R} \right)$ of symmetric $r \times r$ matrices (if $d = 1$) or the space $H_r \left( \mathbb{C} \right)$ of Hermitian $r \times r$ matrices (if $d = 2$) is denoted as $E_r$.
We define the differential operator $\pd{\xi}$ as $\pd{\xi} \left\langle \xi \;\middle|\; x \right\rangle = x$ for $x \in E_r$.
The vectorization of $\pd{\xi}$ is denoted by $\cancel{\partial}_{\xi}$ (i.e., $\cancel{\partial}_{\xi}$ is a column vector of differential operators of length $n = r \left( r - 1 \right) d / 2 + r$, whereas $\pd{\xi}$ is an $r \times r$ matrix of differential operators).
Moreover, the transpose of the column vector $\cancel{\partial}_{\xi}$ is denote by ${\cancel{\partial}_{\xi}}^*$, that is,
${\cancel{\partial}_{\xi}}^*$ is a row vector of differential operators of length $n$.
For example, if $d = 1$ and $r = 2$, we have
$
  \pd{\xi}
  :=
  \begin{bmatrix}
    \frac{\partial}{\partial \xi_{11}} & \frac{1}{2} \frac{\partial}{\partial \xi_{12}} \\
    \frac{1}{2} \frac{\partial}{\partial \xi_{21}} & \frac{\partial}{\partial \xi_{22}}
  \end{bmatrix}
$,
$
  \cancel{\partial}_{\xi}
  :=
  \begin{bmatrix}
    \frac{\partial}{\partial \xi_{11}} \\
    \frac{\partial}{\partial \xi_{12}} \\
    \frac{\partial}{\partial \xi_{22}}
  \end{bmatrix}
$, and
$
  {\cancel{\partial}_{\xi}}^*
  :=
  \begin{bmatrix}
    \frac{\partial}{\partial \xi_{11}} &
    \frac{\partial}{\partial \xi_{12}} &
    \frac{\partial}{\partial \xi_{22}}
  \end{bmatrix}\,.
$
Note that we have $\frac{\partial}{\partial \xi_{12}} = \frac{\partial}{\partial \xi_{21}}$ for this case because $\xi_{12} = \xi_{21}$.
The Fisher information matrix of the family $\{ W_r \left( 1 , \xi\right) \}_{\xi \in \Xi}$ of Wishart distributions is denoted by $H \left( \xi \right)$, where we assume the unit of observation (i.e., $\mu = 1$).
We write $\overline{H} \left( \xi \right)$ if the Fisher information matrix $H \left( \xi \right)$ is represented with respect to the vectorization of the parameter $\xi \in \Xi$ (i.e.,
$H \left( \xi \right)$ is a linear operator from $E_r$ to $E_r$, whereas $\overline{H} \left( \xi \right)$ is an $n \times n$ matrix).

The Laplace--Beltrami operator $\LB$ is the differential operator defined as
\begin{align*}
  \LB K
  :=
  \left| \overline{H} \left( \xi \right) \right|^{-\frac{1}{2}}
  {\cancel{\partial}_{\xi}}^*
  \left(
    \left| \overline{H} \left( \xi \right) \right|^{\frac{1}{2}}
    \left( \overline{H} \left( \xi \right) \right)^{-1}
    {\cancel{\partial}_{\xi}}
    K
  \right)
\end{align*}
for a scalar function $K$ defined on the parameter space $\Xi$.
The Laplace--Beltrami operator does not depend on the choice of parameterizations (see \cite{jost2008riemannian} for the mathematical details).

We focus on the Laplace--Beltrami operator on the parameter space $\Phi^{(i)}$.
Recall that the Fisher information matrix $H \left( \phi \right)$ is of the form $\diag \left( H^{(1)} \left( \phi_{(1)} \right), \ldots, H^{(h)} \left( \phi_{(h)} \right) \right)$.
The Fisher information matrix $H_{(i)} \left( \phi_{(i)} \right)$ with respect to the parameter $\phi_{(i)} = \left( \phi_{(i-1)}, \phi^{(i)} \right) = \left( \zeta_{i-1}, \xi^{(i)}_{1/2}, \xi^{(i)}_0 \right)$ is calculated as
\begin{align*}
  H_{(i)} \left( \phi_{(i)} \right)
  = \begin{bmatrix}
    H^{(i)}_1 \left( \phi_{(i-1)} \right) & 0 & 0 \\
    0 & H^{(i)}_{0A} \left( \phi_{(i-1)}, \phi^{(i)} \right) & H^{(i)}_{0B} \left( \phi_{(i-1)}, \phi^{(i)} \right) \\
    0 & H^{(i)}_{0C} \left( \phi_{(i-1)}, \phi^{(i)} \right) & H^{(i)}_{0D} \left( \phi_{(i-1)}, \phi^{(i)} \right)
  \end{bmatrix}
  \,,
\end{align*}
where
\begin{align*}
  &
  H^{(i)}_1 \left( \phi_{(i-1)} \right) \colon
  w_{1} \mapsto \left( \zeta_{i-1} \right)^{-1} w_1 \left( \zeta_{i-1} \right)^{-1}
  \,,\;
  H^{(i)}_{0A} \left( \phi_{(i-1)}, \phi^{(i)} \right) \colon
  z_{1/2} \mapsto \frac{1}{2} \left( \zeta_{i-1} \right)^{-1} \, z_{1/2} \, \left( \xi^{(i)}_0 \right)^{-1}
  \,,\\
  &
  H^{(i)}_{0B} \left( \phi_{(i-1)}, \phi^{(i)} \right)  \colon
  z_{0} \mapsto - \left( \zeta_{i-1} \right)^{-1} \, \xi^{(i)}_{1/2} \, \left( \xi^{(i)}_0 \right)^{-1} \, z_0 \, \left( \xi^{(i)}_0 \right)^{-1}
  \,,\\
  &
  H^{(i)}_{0C} \left( \phi_{(i-1)}, \phi^{(i)} \right) \colon
  z_{1/2} \mapsto - \frac{1}{2} \left( \xi^{(i)}_0 \right)^{-1}
  \left(
    \left( z_{1/2} \right)^* \, \left( \zeta^{(i)}_1 \right)^{-1} \, \xi^{(i)}_{1/2}
    + \left( \xi^{(i)}_{1/2} \right)^* \, \left( \zeta_{i-1} \right)^{-1} \, z_{1/2}
  \right)
  \left( \xi^{(i)}_0 \right)^{-1}
  \,,\\
  &
  H^{(i)}_{0D} \left( \phi_{(i-1)}, \phi^{(i)} \right)
  z_{0} \mapsto \left( \xi^{(i)}_0 \right)^{-1} \left( z_0 + \left( \xi^{(i)}_{1/2} \right)^* \, \left( \zeta_{i-1} \right)^{-1} \, \xi^{(i)}_{1/2} \, \left( \xi^{(i)}_0 \right)^{-1} \, z_0
  + z_0 \, \left( \xi^{(i)}_0 \right)^{-1} \, \left( \xi^{(i)}_{1/2} \right)^* \, \left( \zeta_{i-1} \right)^{-1} \, \xi^{(i)}_{1/2} \right) \left( \xi^{(i)}_0 \right)^{-1}
\end{align*}
for $w = (w_1, z_{1/2}, z_{0}) \in E_{r_{(i-1)}} \times \Phi^{(i)} = \Phi_{(i)}$,
and the parameter $\phi_{(i-1)}$ is identified with the parameter $\zeta_{i-1} \in E_{r_{(i-1)}}$.
The inverse $V_{(i)} \left( \phi_{(i)} \right)$ of the Fisher information matrix $H_{(i)} \left( \phi_{(i)} \right)$ is calculated as
\begin{align*}
  V_{(i)} \left( \phi_{(i)} \right)
  = \begin{bmatrix}
    V^{(i)}_1 \left( \phi_{(i-1)} \right) & 0 & 0 \\
    0 & V^{(i)}_{0A} \left( \phi_{(i-1)}, \phi^{(i)} \right) & V^{(i)}_{0B} \left( \phi_{(i-1)}, \phi^{(i)} \right) \\
    0 & V^{(i)}_{0C} \left( \phi_{(i-1)}, \phi^{(i)} \right) & V^{(i)}_{0D} \left( \phi_{(i-1)}, \phi^{(i)} \right)
  \end{bmatrix}
  \,,
\end{align*}
where
\begin{align*}
  &
  V^{(i)}_1 \left( \phi_{(i-1)} \right) \colon
  w_{1} \mapsto \zeta_{i-1} \, w_1 \, \zeta_{i-1}
  \,,\\
  &
  V^{(i)}_{0A} \left( \phi_{(i-1)}, \phi^{(i)} \right) \colon
  z_{1/2} \mapsto 2 \, \zeta_{i-1} \left( z_{1/2} \right) \xi^{(i)}_0 + 2 \, \xi^{(i)}_{1/2} \left( z_{1/2} \right)^{*} \xi^{(i)}_{1/2}.
  + 2 \, \xi^{(i)}_{1/2} \left( \xi^{(i)}_0 \right)^{-1} \left( \xi^{(i)}_{1/2} \right)^* z_{1/2} \, \xi^{(i)}_0
  \,,\\
  &
  V^{(i)}_{0B} \left( \phi_{(i-1)}, \phi^{(i)} \right) \colon
  z_{0} \mapsto 2 \, \xi^{(i)}_{1/2} \, z_0 \, \xi^{(i)}_0
  \,,\;
  V^{(i)}_{0C} \left( \phi_{(i-1)}, \phi^{(i)} \right) \colon
  z_{1/2} \mapsto \xi^{(i)}_0 \left( z_{1/2} \right)^* \, \xi^{(i)}_{1/2}  + \left( \xi^{(i)}_{1/2} \right)^* z_{1/2} \, \xi^{(i)}_{0}
  \,,\\
  &
  V^{(i)}_{0D} \left( \phi_{(i-1)}, \phi^{(i)} \right)
  \colon z_{0} \mapsto \xi^{(i)}_0 \, z_0 \, \xi^{(i)}_0
\end{align*}
for $w = (w_1, z_{1/2}, z_{0}) \in E_{r_{(i-1)}} \times \Phi^{(i)} = \Phi_{(i)}$.
We write $\overline{V} \left( \xi \right)$ if the inverse $V \left( \xi \right)$ of the Fisher information matrix $H \left( \xi \right)$ is represented with respect to the vectorization of the parameter $\xi \in \Xi$ (i.e., the $n \times n$ matrix $\overline{V} \left( \xi \right)$ is the inverse of the $n \times n$ matrix $\overline{H} \left( \xi \right)$).
Similarly, we define the matrices $\overline{V}^{(i)}_1$, $\overline{V}^{(i)}_{0A}$, $\overline{V}^{(i)}_{0B}$, $\overline{V}^{(i)}_{0C}$, and $\overline{V}^{(i)}_{0D}$ with respect to the vectorization of the parameter $\phi_{(i)} = \left( \zeta_{i-1}, \xi^{(i)}_{1/2}, \xi^{(i)}_0 \right)$.

\begin{proof}[\textbf{\upshape Proof of Proposition \ref{proposition:main_asymptotic}:}]
We calculate the eigenvalue of the Laplace--Beltrami on the parameter space $\Phi^{(i)}$ as follows. The scalar function (\ref{definition:K_general}) with $t^{(1)} = \cdots = t^{(i-1)} = t$ and $t^{(i+1)} = \cdots = t^{(h)} = - n / r$ is considered.
Then, we have
\begin{align}
  \left. \left( \LB K_{ \mathbf{t} } \right) \middle\slash K_{ \mathbf{t} } \right.
  &=
  \left.
    \left(
      \left| \overline{H} \right|^{-\frac{1}{2}}
      \cancel{\partial}_{\zeta_{i-1}}^{\;*}
      \left(
        \left| \overline{H} \right|^{\frac{1}{2}}
        \overline{V}^{(i)}_1
        \,
        \cancel{\partial}_{\zeta_{i-1}}
        K_{ \mathbf{t}}
      \right)
    \right)
  \middle\slash
    K_{ \mathbf{t}}
  \right.
  \nonumber\\
  &\quad+
  \left.
    \left(
      \left| \overline{H} \right|^{-\frac{1}{2}}
      \cancel{\partial}_{\xi^{(i)}_{1/2}}^{\;*}
      \left(
        \left| \overline{H} \right|^{\frac{1}{2}}
        \overline{V}^{(i)}_{0A}
        \,
        \cancel{\partial}_{\xi^{(i)}_{1/2}}
        K_{ \mathbf{t}}
      \right)
    \right)
  \middle\slash
    K_{ \mathbf{t}}
  +
    \left(
      \left| \overline{H} \right|^{-\frac{1}{2}}
      \cancel{\partial}_{\xi^{(i)}_{1/2}}^{\;*}
      \left(
        \left| \overline{H} \right|^{\frac{1}{2}}
        \overline{V}^{(i)}_{0B}
        \,
        \cancel{\partial}_{\xi^{(i)}_{0}}
        K_{ \mathbf{t}}
      \right)
    \right)
  \middle\slash
    K_{ \mathbf{t}}
  \right.
  \nonumber\\
  &\quad+
  \left.
    \left(
      \left| \overline{H} \right|^{-\frac{1}{2}}
      \cancel{\partial}_{\xi^{(i)}_{0}}^{\;*}
      \left(
        \left| \overline{H} \right|^{\frac{1}{2}}
        \overline{V}^{(i)}_{0C}
        \,
        \cancel{\partial}_{\xi^{(i)}_{1/2}}
        K_{ \mathbf{t}}
      \right)
    \right)
  \middle\slash
    K_{ \mathbf{t}}
  +
    \left(
      \left| \overline{H} \right|^{-\frac{1}{2}}
      \cancel{\partial}_{\xi^{(i)}_0}^{\;*}
      \left(
        \left| \overline{H} \right|^{\frac{1}{2}}
        \overline{V}^{(i)}_{0D}
        \,
        \cancel{\partial}_{\xi^{(i)}_0}
        K_{ \mathbf{t}}
      \right)
    \right)
  \middle\slash
    K_{ \mathbf{t}}
  \right.
  \label{lemma:LB}
\end{align}
for $\mathbf{t} = \left( t, \ldots, t, t^{(i)}, - n / r, \ldots, - n / r \right) \in \mathbb{R}^h$,
where
\begin{align*}
  \left.
    \left(
      \left| \overline{H} \right|^{-\frac{1}{2}}
      \cancel{\partial}_{\zeta_{i-1}}^{\;*}
      \left(
        \left| \overline{H} \right|^{\frac{1}{2}}
        \overline{V}^{(i)}_1
        \,
        \cancel{\partial}_{\zeta_{i-1}}
        K_{ \mathbf{t}}
      \right)
    \right)
  \middle\slash
    K_{ \mathbf{t}}
  \right.
  &=
  r_{(i-1)}
  \left(
    \left( \frac{1}{2} \left( t + \frac{n}{r} \right) \right)^2
    -
    \left( \frac{n}{r} - \left( \frac{r_{(i-1)}}{2} d + 1 \right) \right)
    \left( \frac{1}{2} \left( t + \frac{n}{r} \right) \right)
  \right)
  \,,\\
  \left.
    \left(
      \left| \overline{H} \right|^{-\frac{1}{2}}
      \cancel{\partial}_{\xi^{(i)}_{1/2}}^{\;*}
      \left(
        \left| \overline{H} \right|^{\frac{1}{2}}
        \overline{V}^{(i)}_{0A}
        \,
        \cancel{\partial}_{\xi^{(i)}_{1/2}}
        K_{ \mathbf{t}}
      \right)
    \right)
  \middle\slash
    K_{ \mathbf{t}}
  \right.  
  &=
  0
  \,,\\
  \left.
    \left(
      \left| \overline{H} \right|^{-\frac{1}{2}}
      \cancel{\partial}_{\xi^{(i)}_{1/2}}^{\;*}
      \left(
        \left| \overline{H} \right|^{\frac{1}{2}}
        \overline{V}^{(i)}_{0B}
        \,
        \cancel{\partial}_{\xi^{(i)}_{0}}
        K_{ \mathbf{t}}
      \right)
    \right)
  \middle\slash
    K_{ \mathbf{t}}
  \right.
  &=
  r_{(i-1)} k^{(i)} d \left( \frac{1}{2} \left( t^{(i)} + \frac{n}{r} \right) \right)
  \,,\\
  \left.
    \left(
      \left| \overline{H} \right|^{-\frac{1}{2}}
      \cancel{\partial}_{\xi^{(i)}_{0}}^{\;*}
      \left(
        \left| \overline{H} \right|^{\frac{1}{2}}
        \overline{V}^{(i)}_{0C}
        \,
        \cancel{\partial}_{\xi^{(i)}_{1/2}}
        K_{ \mathbf{t}}
      \right)
    \right)
  \middle\slash
    K_{ \mathbf{t}}
  \right.
  &=
  0
  \,,\\
  \left.
    \left(
      \left| \overline{H} \right|^{-\frac{1}{2}}
      \cancel{\partial}_{\xi^{(i)}_0}^{\;*}
      \left(
        \left| \overline{H} \right|^{\frac{1}{2}}
        \overline{V}^{(i)}_{0D}
        \,
        \cancel{\partial}_{\xi^{(i)}_0}
        K_{ \mathbf{t}}
      \right)
    \right)
  \middle\slash
    K_{ \mathbf{t}}
  \right.
  &=
  k^{(i)}
  \left(
    \left( \frac{1}{2} \left( t^{(i)} + \frac{n}{r} \right) \right)^2
    -
    \left( \frac{n}{r} - \left( \frac{ k^{(i)} - 1 }{2} d + 1 \right) \right)
    \left( \frac{1}{2} \left( t^{(i)} + \frac{n}{r} \right) \right)
  \right)
  \,.
\end{align*}
Therefore, the terms that are relevant to $t^{(i)} \in \mathbb{R}$ in the eigenvalue (\ref{lemma:LB}) are summarized as
\begin{align*}
  &
  k^{(i)}
  \left(
    \left( \frac{1}{2} \left( t^{(i)} + \frac{n}{r} \right) \right)^2
    -
    \left( t_R^{(i)} + \frac{n}{r} \right) 
    \left( \frac{1}{2} \left( t^{(i)} + \frac{n}{r} \right) \right)
  \right)
  =
  \frac{k^{(i)}}{4} \left( t^{(i)} - t_R^{(i)} \right)^2
  -
  \frac{d^2}{16} k^{(i)}
  \left(
    r - 2 r_{(i)} + k^{(i)}
  \right)^2
  \,.
\end{align*}
\end{proof}

\section{Enriched standard conjugate prior distribution}
\label{section:ESCPD}

The closure (i.e., positive semi-definite matrices) of $E_r^+$ in $E_r$ is denoted by $\overline{E_r^+}$.
The family of enriched standard conjugate prior distributions for the family $\{ W_r \left( \mu, \phi \right) \}_{\phi \in \Phi}$ of Wishart distributions is defined as
\begin{align}
  \pi_{\mathbf{s}, \mathbf{t}} \left( \phi \right) d\phi
  :=
  \prod_{i=1}^h
  \pi^{(i)}_{s^{(i)}, t^{(i)}} \left( \phi^{(i)} \right) d\phi^{(i)} 
  \label{definition:enriched_standard_conjugate_prior}
\end{align}
for the hyperparameters $\mathbf{s} := \left( s^{(1)}, \ldots, s^{(h)} \right) \in E_{r_{(1)}} \times \cdots \times E_{r_{(h)}}$ and $\mathbf{t} := \left( t^{(1)}, \ldots, t^{(h)} \right) \in \mathbb{R}^h$.
Each prior distribution $\pi^{(i)}_{s^{(i)}, t^{(i)}} \left( \phi^{(i)} \right) d\phi^{(i)}$ on the parameter space $\Phi^{(i)}$ is defined as
\begin{align}
  &
  \pi^{(i)}_{s^{(i)}, t^{(i)}} \left( \phi^{(i)} \right) d\phi^{(i)} 
  = 
  \left| \xi^{(i)}_0 \right|^{ t^{(i)} }
  \exp
  \left(
    - \left\langle  \xi^{(i)}_{1/2} \, {\xi^{(i)}_0}^{-1} \, {\xi^{(i)}_{1/2}}^* \;\middle|\; s^{(2)}_1 \right\rangle
    - \left\langle \xi^{(i)}_{1/2} \;\middle|\; s^{(2)}_{1/2} \right\rangle
    - \left\langle \xi^{(i)}_0 \;\middle|\; s^{(2)}_0 \right\rangle
  \right)
  d\xi^{(i)}_{1/2}
  d\xi^{(i)}_0
  \label{definition:enriched_standard_conjugate_prior_i}
  \,,
\end{align}
where we write
$
  s^{(i)}
  =
  \begin{bmatrix}
    s^{(i)}_1 & s^{(i)}_{1/2} \\
    { s^{(i)}_{1/2} }^* & s^{(i)}_{0}
  \end{bmatrix}
  \in E_{r_{(i)}}
$
with respect to the partition $\left( r_{(i-1)}, k^{(i)} \right)$ of rank $r_{(i)}$ for $i \in \{ 1, \ldots, h \}$ (see \cite{consonni2003enriched}).
Note that $d\phi = \prod_{i=1}^h d\phi^{(i)}$ and $\left| \xi \right| = \prod_{i=1}^h \left| \xi^{(i)}_0 \right|$,
where $\xi^{(1)}_0$ is interpreted as $\zeta_{(1)}$.
The enriched standard conjugate prior distribution (\ref{definition:enriched_standard_conjugate_prior}) reduces to the usual standard conjugate prior distributions if $s^{(1)} = \cdots = s^{(h)} = 0$ and $t^{(1)} = \cdots = t^{(h)} = t$.
The prior distribution (\ref{definition:enriched_standard_conjugate_prior_i}) is proper if $s^{(i)} \in E_r^+$ and $t^{(i)} > - \left( \left( r_{(i)} - k^{(i)} \right) / 2 \right) d - 1$.

\begin{lemma}
  The normalization constant of the prior distribution $\pi^{(i)}_{s^{(i)}, t^{(i)}} \left( \phi^{(i)} \right) d\phi^{(i)}$ on the parameter space $\Phi^{(i)}$ is calculated as
  \begin{align}
    \int_{\Phi^{(i)}} \pi^{(i)}_{s^{(i)}, t^{(i)}} \left( \phi^{(i)} \right) \,d\phi^{(i)}
    &=
    \pi^{\frac{ r_{(i-1)} k^{(i)} }{ 2 } d } \,
    \Gamma_{k^{(i)}} \left( t^{(i)} + \frac{n_{(i)}}{r_{(i)}} \right)
    \left| s^{(i)}_1 \right|^{ t^{(i)} + \frac{n_{(i-1)}}{r_{(i-1)}} }
    \left| s^{(i)} \right|^{ - \left( t^{(i)} + \frac{n_{(i)}}{r_{(i)}} \right)}
    \label{lemma:normalization_constant}
  \end{align}
  if $s^{(i)} \in E_{r_{(i)}}^+$ and $t^{(i)} > - \left( \left( r_{(i)} - k^{(i)} \right) / 2 \right) d - 1$.
  If $i = 1$, then (\ref{lemma:normalization_constant}) is interpreted as
  \begin{align*}
    \int_{\Phi^{(1)}} \pi^{(1)}_{s^{(1)}, t^{(1)}} \left( \phi^{(1)} \right) \,d\phi^{(1)}
    &=
    \Gamma_{k^{(1)}} \left( t^{(1)} + \frac{n_{(1)}}{r_{(1)}} \right)
    \left| s^{(1)} \right|^{ - \left( t^{(1)} + \frac{n_{(1)}}{r_{(1)}} \right)}
  \end{align*}
  for $s^{(1)} \in E_{r_{(1)}}^+$ and $t^{(1)} > -1$.
\end{lemma}

\begin{proof}[\textbf{\upshape Proof:}]
  We have
  \begin{align*}
    \int_{\Phi^{(i)}}
    \pi^{(i)}_{s^{(i)}, t^{(i)}} \left( \phi^{(i)} \right)
    \,d\phi^{(i)}
    &=
    \int \left| \xi^{(i)}_0 \right|^{ t^{(i)} }
    \exp \left(
      - \left\langle \xi^{(i)}_0 \;\middle|\; s^{(i)}_0 \right\rangle
    \right)
    \left(
      \int
      \exp \left(
        - \left\langle {\xi^{(i)}_{1/2}} \, \left( \xi^{(i)}_0 \right)^{-1} \, {\xi^{(i)}_{1/2}}^* \;\middle|\; s^{(i)}_1 \right\rangle
        - \left\langle \xi^{(i)}_{1/2} \;\middle|\; s^{(i)}_{1/2} \right\rangle
      \right)
      \,d\xi^{(i)}_{1/2}
    \right)
    \,d\xi^{(i)}_{0}
    \nonumber\\
    &=
    \pi^{\frac{ \left( r_{(i)} - k^{(i)} \right) k^{(i)} }{ 2 } d }
    \left| s^{(i)}_1 \right|^{ - \frac{ k^{(i)} }{ 2 } d }
    \int \left| \xi^{(i)}_0 \right|^{ t^{(i)} + \frac{ \left( r_{(i)} - k^{(i)} \right) }{ 2 } d }
    \exp \left(
      - \left\langle
        \xi^{(i)}_0
      \;\middle|\;
        s^{(i)}_0 - {s^{(i)}_{1/2}}^* \, \left( s^{(i)}_1 \right)^{-1} \, {s^{(i)}_{1/2}}
      \right\rangle
    \right)
    \,d\xi^{(i)}_{0}
    \nonumber\\
    &=
    \pi^{\frac{ \left( r_{(i)} - k^{(i)} \right) k^{(i)} }{ 2 } d }
    \left| s^{(i)}_1 \right|^{ - \frac{ k^{(i)} }{ 2 } d }
    \left| s^{(i)}_0 - {s^{(i)}_{1/2}}^* \, \left( s^{(i)}_1 \right)^{-1} \, {s^{(i)}_{1/2}} \right|^{ - \left( t^{(i)} + \left( \frac{ r_{(i)} - 1 }{2} d + 1 \right) \right)}
    \Gamma_{k^{(i)}} \left( t^{(i)} + \left( \frac{ r_{(i)} - 1 }{2} d + 1 \right) \right)
    \nonumber \\
    &=
    \pi^{\frac{ \left( r_{(i)} - k^{(i)} \right) k^{(i)} }{2}d} \,
    \Gamma_{k^{(i)}} \left( t^{(i)} + \left( \frac{ r_{(i)} - 1 }{2} d + 1 \right) \right)
    \left| s^{(i)}_1 \right|^{ t^{(2)} + \left( \frac{ r_{(i)} - k^{(i)} - 1 }{2} d + 1 \right)  }
    \left| s^{(i)} \right|^{ - \left( t^{(i)} + \left( \frac{ r_{(i)} - 1 }{2} d + 1 \right) \right)}
  \end{align*}
  for $s^{(i)} \in E_{r_{(i)}}^+$ and $t^{(i)} > - \left( \left( r_{(i)} - k^{(i)} \right) / 2 \right) d - 1$, where we have used the equation
  \begin{align*}
    \int
      \exp
      \left(
        \tr
        \left(
          - V^{-1} \left( z - M \right) U^{-1} \left( z - M \right)^T
        \right)
      \right)
    dz
    =
    \pi^{\frac{mn}{2}d}
    \left| V \right|^{\frac{n}{2}d}
    \left| U \right|^{\frac{m}{2}d}
  \end{align*}
  for $U \in E^+_n$ and $V \in E^+_m$ to calculate
  \begin{align*}
    &
    \int
      \exp \left(
        - \left\langle {\xi^{(i)}_{1/2}} \, \left( \xi^{(i)}_0 \right)^{-1} \, {\xi^{(i)}_{1/2}}^* \;\middle|\; s^{(i)}_1 \right\rangle
        - \left\langle \xi^{(i)}_{1/2} \;\middle|\; s^{(i)}_{1/2} \right\rangle
      \right)
    \,d\xi^{(i)}_{1/2}
    \nonumber \\
    &\qquad=
    \int
    \exp \left(
      - \tr \left(
        \left( \xi^{(i)}_0 \right)^{-1} \left( \xi^{(i)}_{1/2} - M \right)^T s^{(i)}_1 \left( \xi^{(i)}_{1/2} - M \right)
      \right)
      + \tr \left(
        \left( s^{(i)}_{1/2} \right)^* \left( s^{(i)}_1 \right)^{-1} \left( s^{(i)}_{1/2} \right) \xi^{(i)}_0
      \right)
    \right)
    \,d\xi^{(i)}_{1/2}
    \nonumber \\
    &\qquad=
    \pi^{\frac{ \left( r_{(i)} - k^{(i)} \right) k^{(i)} }{ 2 } d } \left| \xi^{(i)}_0 \right|^{ \frac{ r_{(i)} - k^{(i)} }{ 2 } d } \left| s^{(i)}_1 \right|^{ - \frac{ k^{(i)} }{ 2 } d }
    \exp \left(\left\langle
      \xi^{(i)}_0
    \;\middle|\;
      {s^{(i)}_{1/2}}^* \, \left( s^{(i)}_1 \right)^{-1} \, {s^{(i)}_{1/2}}
    \right\rangle\right)
  \end{align*}
  with an $\left( r_{(i)} - k^{(i)} \right) \times k^{(i)}$ matrix $M := - \left( s^{(i)}_1 \right)^{-1} s^{(i)}_{1/2} \; \xi^{(i)}_0$.
\end{proof}

\begin{lemma}
  Let $\mu > \left( r - 1 \right) d / 2$ and $\nu > \left( r - 1 \right) d / 2$.
  The risk $R^{\mu, \nu} \left( \xi, \delta^{\mu, \nu}_{\mathbf{s}, \mathbf{t}} \right)$ of the Bayesian predictive distribution $\delta^{\mu, \nu}_{\mathbf{s}, \mathbf{t}}$ based on the enriched standard conjugate prior distribution (\ref{definition:enriched_standard_conjugate_prior}) is calculated as
  \begin{align*}
    R^{\mu, \nu} \left( \xi, \delta^{\mu, \nu}_{\mathbf{s}, \mathbf{t}} \right)
    =
    \sum_{i=1}^h
    R^{(i) \mu, \nu} \left( \xi, \delta^{(i) \mu, \nu}_{s^{(i)}, t^{(i)}} \right)
    \,,
  \end{align*}
  where
  \begin{align}
    R^{ (i) \mu, \nu } \left( \phi^{(i)}, \delta^{ (i) \mu, \nu }_{s^{(i)}, t^{(i)}} \right)
    &=
    - \nu k^{(i)}
    + \nu \log \left| \xi^{(i)}_0 \right|
    - \log \Gamma_{k^{(i)}} \left( t^{(i)} + \mu + \nu + \frac{n_{(i)}}{r_{(i)}} \right)
    + \log \Gamma_{k^{(i)}} \left( t^{(i)} + \mu + \frac{n_{(i)}}{r_{(i)}} \right)
    \nonumber\\
    &\quad
    + \left( t^{(i)} + \mu + \nu + \frac{n_{(i)}}{r_{(i)}} \right)
    E \left[ \log \left| s^{(i)} + X_{(i)} + Y_{(i)} \right| \right]
    - \left( t^{(i)} + \mu + \frac{n_{(i)}}{r_{(i)}} \right)
    E \left[ \log \left| s^{(i)} + X_{(i)} \right| \right]
    \nonumber\\
    &\quad
    - \left( t^{(i)} + \mu + \nu + \frac{n_{(i-1)}}{r_{(i-1)}} \right)
    E \left[ \log \left| s^{(i)}_1 + X^{(i)}_1 + Y^{(i)}_1 \right| \right]
    + \left( t^{(i)} + \mu + \frac{n_{(i-1)}}{r_{(i-1)}} \right)
    E \left[ \log \left| s^{(i)}_1 + X^{(i)}_1 \right| \right]
    \label{lemma:R_i}
  \end{align}
  for $s^{(i)} \in \overline{E_{r_{(i)}}^+}$, $t^{(i)} > - \mu - \left( \left( r_{(i)} - k^{(i)} \right) / 2 \right) d - 1$,
  and $E$ denotes the expectation over the distributions $X \sim W_r \left( \mu, \phi \right)$ and $Y \sim W_r \left( \nu, \phi \right)$.
  The last two terms in (\ref{lemma:R_i}) are interpreted as $0$ if $i = 1$.
\end{lemma}

\begin{proof}[\textbf{\upshape Proof:}]
  The Bayesian predictive distribution $\delta^{\mu, \nu}_{\mathbf{s}, \mathbf{t}} \left( y \;\middle|\; x \right) dy$ based on the prior distribution (\ref{definition:enriched_standard_conjugate_prior}) is calculated as the product
  \begin{align*}
    \delta^{\mu, \nu}_{\mathbf{s}, \mathbf{t}} \left( y \;\middle|\; x \right) dy
    =
    \prod_{i=1}^h
    \delta^{(i) \mu,\nu}_{s^{(i)}, t^{(i)}} \left( y^{(i)} \;\middle|\; x_{(i)}, y_{(i-1)} \right) dy^{(i)}
  \end{align*}
  of the conditional Bayesian predictive distributions
  \begin{align}
    &
    \delta^{(i) \mu,\nu}_{s^{(i)}, t^{(i)}} \left( y^{(i)} \;\middle|\; x_{(i)}, y_{(i-1)} \right) dy^{(i)}
    :=
    \left(
      \int_{\Phi^{(i)}}
      p^{(i) \nu} \left( y^{(i)} \;\middle|\; y_{(i-1)}, \phi^{(i)} \right)
      \pi^{(i) \mu}_{s^{(i)}, t^{(i)}} \left( \phi^{(i)} \;\middle|\; x_{(i)} \right)
      d\phi^{(i)}
    \right)
    dy^{(i)}
    \label{definition:CBPD_ESCP}
  \end{align}
  for $s^{(i)} \in \overline{E_{r_{(i)}}^+}$, $t^{(i)} > - \mu - \left( \left( r_{(i)} - k^{(i)} \right) / 2 \right) d - 1$.
  Direct computation yields 
  \begin{align*}
    - 
    \log
    \frac{
      \delta^{\mu,\nu}_{s^{(i)}, t^{(i)}} \left( y^{(i)} \;\middle|\; x_{(i)}, y_{(i-1)} \right)
    }{
      p^{\nu(i)} \left( y^{(i)} \;\middle|\; y_{(i-1)}, \phi^{(i)} \right)
    }
    &=
    \nu \log \left| \xi^{(i)}_0 \right|
    -
    \left(
      \left\langle \zeta_i \;\middle|\; y_{(i)} \right\rangle
      -
      \left\langle \zeta_{i-1} \;\middle|\; y^{(i)}_1 \right\rangle
    \right)
    - \log \Gamma_{k^{(i)}} \left( t^{(i)} + \mu + \nu + \frac{n_{(i)}}{r_{(i)}} \right)
    + \log \Gamma_{k^{(i)}} \left( t^{(i)} + \mu + \frac{n_{(i)}}{r_{(i)}} \right)
    \nonumber\\
    &\quad
    + \left( t^{(i)} + \mu + \nu + \frac{n_{(i)}}{r_{(i)}} \right) \log \left| s^{(i)} + x_{(i)} + y_{(i)} \right|
    - \left( t^{(i)} + \mu + \frac{n_{(i)}}{r_{(i)}} \right) \log \left| s^{(i)} + x_{(i)} \right|
    \nonumber\\
    &\quad
    - \left( t^{(i)} + \mu + \nu + \frac{n_{(i-1)}}{r_{(i-1)}} \right) \log \left| s^{(i)}_1 + x^{(i)}_1 + y^{(i)}_1 \right|
    + \left( t^{(i)} + \mu + \frac{n_{(i-1)}}{r_{(i-1)}} \right) \log \left| s^{(i)}_1 + x^{(i)}_1 \right|
    \,.
  \end{align*}  
  Recall that we have $y_{(i)} = \left( y_{(i-1)}, y^{(i)} \right)$, and $y^{(i)}_1 = y_{(i-1)}$.
  $E \left[ Y^{(i)} \right] = \nu \left( \zeta_{i} \right)^{-1}$, and $E \left[ Y^{(i)}_1 \right] = E \left[ Y^{(i-1)} \right] = \nu \left( \zeta_{i-1} \right)^{-1}$ are used to compute the risk
  \begin{align*}
    R^{ (i) \mu, \nu } \left( \phi^{(i)}, \delta^{ (i) \mu, \nu }_{s^{(i)}, t^{(i)}} \right)
    :=
    E \left[
      - 
      \log
      \frac{
        \delta^{\mu,\nu}_{s^{(i)}, t^{(i)}} \left( Y^{(i)} \;\middle|\; X_{(i)}, Y_{(i-1)} \right)
      }{
        p^{(i)\nu} \left( Y^{(i)} \;\middle|\; Y_{(i-1)}, \phi^{(i)} \right)
      }
    \right]
  \end{align*}
  of the conditional Bayesian predictive distribution (\ref{definition:CBPD_ESCP}).
\end{proof}

\begin{proof}[\textbf{\upshape Proof of Proposition \ref{proposition:main_exact}:}]
Note that $X_{(i)} + Y_{(i)} \sim W_{r_{(i)}} \left( \mu + \nu, \zeta_i \right)$, $X_{(i)} \sim W_{r_{(i)}} \left( \mu, \zeta_i \right)$, $X^{(i)}_1 + Y^{(i)}_1 \sim W_{r_{(i-1)}} \left( \mu + \nu, \zeta_{i-1} \right)$, and $X^{(i)}_1  \sim W_{r_{(i-1)}} \left( \mu, \zeta_{i-1} \right)$ if $X \sim W_r \left( \mu, \phi \right)$ and $Y \sim W_r \left( \nu, \phi \right)$.
We see that (\ref{lemma:R_i}) with $\mathbf{s} = 0$ yields the proof because we have $E \left[ X_{(i)} + Y_{(i)} \right] = \psi_{r_{(i)}} \left( \mu + \nu \right) - \log \left| \zeta_i \right|$, $ E \left[ X_{(i)} \right] = \psi_{r_{(i)}} \left( \mu \right) - \log \left| \zeta_i \right|$, $E \left[ X^{(i)}_1 + Y^{(i)}_1 \right] = \psi_{r_{(i-1)}} \left( \mu + \nu \right) - \log \left| \zeta_{i-1} \right|$, $E \left[ X^{(i)}_1 \right] = \psi_{r_{(i-1)}} \left( \mu \right) - \log \left| \zeta_{i-1} \right|$, and $\left| \zeta_i \right| = \left| \zeta_{i-1} \right| \left| \xi^{(i)}_0 \right|$.
\end{proof}

\section{Multivariate polygamma function}
\label{section:MPF}

There exists a finite term expansion of the log-gamma function
\begin{align}
	\log \Gamma (z) &= \left( z - \frac{1}{2}  \right) \log z - z + \frac{1}{2} \log 2 \pi
	+ 2 \int_0^{\infty} \frac{\arctan (t/z)}{e^{2 \pi t} - 1} \,dt
	\nonumber \\
	&= \left( z - \frac{1}{2}  \right) \log z - z + \frac{1}{2} \log 2 \pi
	+ \sum_{n=1}^N \frac{(-1)^{n-1} B_{2n}}{2n (2n - 1) z^{2n - 1}}
	+ (-1)^{N} \frac{2}{z^{2N - 1}} \int_0^{\infty} \left( \int_{0}^{t} \frac{u^{2n} \,du}{u^2 + z^2} \right) \frac{dt}{e^{2 \pi t} - 1}
	\label{formula:loggamma_finite_expansion}
\end{align}
and that of the digamma function
\begin{align}
	\psi (z) &= \frac{d}{dz} \log \Gamma (z) = \log z - \frac{1}{2z} - 2 \int_0^{\infty} \frac{t}{(t^2 + z^2)(e^{2 \pi t} - 1)} \,dt
	\nonumber \\
	&= \log z - \frac{1}{2z} - \sum_{n=1}^N \frac{B_{2n}}{2n z^{2n}} + (-1)^{N+1} \frac{2}{z^{2N}} \int_0^{\infty} \frac{t^{2N + 1}}{(t^2 + z^2)(e^{2 \pi t} - 1)} \,dt
	\label{formula:digamma_finite_expansion}
\end{align}
for $\Re (z) > 0$, where $B_{2n}$ are Bernoulli numbers (e.g., $B_0 = 1, B_2 = 1/6, B_4 = - 1/30, \ldots$)(see \cite{whittaker2021course}).
Note that some authors use $B_n$ to denote the Bernoulli number $B_{2n}$.
Using equations (\ref{formula:loggamma_finite_expansion}) and (\ref{formula:digamma_finite_expansion}),
we can calculate the upper and lower bounds of the log-gamma and digamma functions of any order because the integrands in (\ref{formula:loggamma_finite_expansion}) and (\ref{formula:digamma_finite_expansion}) are positive.

An asymptotic expansion of the multivariate log-gamma function
\begin{align}
  \log \Gamma_r \left( \mu + x \right)
  &= r \left( \mu \log \mu \right)
  - r \mu
  +
  \left(
    -
    \frac{1}{4} d r^2
    +
    \frac{1}{4} \left( d - 2 \right) r
    +
    r x
  \right)
  \log \mu
  +
  \left(
    \frac{1}{4} d \left( \log \pi \right) r^2
    +
    \frac{1}{4} \left( - \left( \log \pi \right) d + 2 \left( \log 2 \pi \right) \right) r
  \right)
  \nonumber\\
  &\quad+
  \left(
    \frac{1}{24} d^2 r^3
    -
    \frac{1}{16} \left( d - 2 \right) d r^2
    +
    \frac{1}{48} \left( d^2 - 6 d + 4 \right) r
    +
    \left(
      -
      \frac{1}{4} d r^2
      +
      \frac{1}{4} \left( d - 2 \right) r
    \right)
    x
    +
    \frac{1}{2} r x^2
  \right)
  \mu^{-1}
  \nonumber\\
  &\quad+
  \left(
    \frac{1}{192} d^3 r^4
    -
    \frac{1}{96} \left( d - 2 \right) d^2 r^3
    +
    \frac{1}{192} \left( d^2 - 6 d + 4 \right) d r^2
    +
    \frac{1}{96} \left( d - 2 \right) d r
  \right.
  \nonumber\\
  &\qquad\qquad
  \left.
    +
    \left(
      -
      \frac{1}{24} d^2 r^3
      +
      \frac{1}{16} \left( d - 2 \right) d r^2
      -
      \frac{1}{48} \left( d^2 - 6 d + 4 \right) r
    \right)
    x
  \right.
  \left.
    +
    \left(
      \frac{1}{8} d r^2
      -
      \frac{1}{8} \left( d - 2 \right) r
    \right)
    x^2
    -
    \frac{1}{6} r x^3
  \right)
  \mu^{-2}
  \nonumber\\
  &\quad+
  O \left( \mu^{-3} \right)
  \label{formula:multivariate_loggamma_mu_x}
\end{align}
and that of the multivariate digamma function
\begin{align}
  \psi_r \left( \mu + x \right)
  &= r \log \mu 
  +
  \left(
    -
    \frac{1}{4} d r^2
    +
    \frac{1}{4} \left( d - 2 \right) r
    +
    r x
  \right)
  \mu^{-1}
  \nonumber \\
  &\quad+
  \left(
    -
    \frac{1}{24} d^2 r^3
    +
    \frac{1}{16} \left( d - 2 \right) d r^2
    -
    \frac{1}{48} \left( d^2 - 6 d + 4 \right) r
    +
    \left(
      \frac{1}{4} d r^2
      -
      \frac{1}{4} \left( d - 2 \right) r
    \right)
    x
    -
    \frac{1}{2} r
    x^2
  \right)
  \mu^{-2}
  +
  O \left( \mu^{-3} \right)
  \label{formula:multivariate_polygamma_mu_x}
\end{align}
are helpful in the present study.

\section*{Acknowledgments}

The authors would like to express their gratitude to Keisuke Yano for his helpful comments and suggestions.

\bibliographystyle{myjmva}
\bibliography{report}

\end{document}